\documentclass[11pt,a4paper, reqno]{amsart}
\usepackage{mathrsfs}

\usepackage{amsmath,amsfonts,verbatim}
\usepackage{latexsym}
\usepackage{amssymb,leftidx}
\usepackage{extarrows}
\usepackage{overpic}
\usepackage{color}
\usepackage{epsfig}
\usepackage{subfigure}
\usepackage{tikz}
\usepackage{bbm}
\usepackage{tikz, caption}
\usepackage[font=small,labelfont=bf]{caption}
\usepackage{mathtools}

\usepackage[colorlinks=true]{hyperref}
\hypersetup{urlcolor=blue, citecolor=blue}

\usepackage{amssymb}
\usepackage{mathrsfs}
\usepackage{amscd}
\usepackage{cite}

\newcommand\R{\mathbb{R}}
\newcommand\C{\mathbb{C}}
\newcommand\Z{\mathbb{Z}}

\newcommand\A{{\bf A}}

\newcommand{\norm}[1]{\left\|#1\right\|}

\newcommand{\supp}{\operatorname{supp}}

\newcommand{\dG}{d_{G}}

\newcommand{\T}{\mathbb{T}}

\newcommand\LL{\mathcal{L}}

\usepackage{graphicx}
\usepackage{float}

\numberwithin{equation}{section}
\newtheorem{proposition}{Proposition}[section]

\newtheorem{lemma}{Lemma}[section]
\newtheorem{theorem}{Theorem}[section]

\newtheorem{remark}{Remark}[section]


\makeatletter
\@namedef{subjclassname@2020}{%
  \textup{2020} Mathematics Subject Classification}
\makeatother


\begin{document}
\title[Bochner-Riesz means for critical magnetic Schr\"odinger operators]{Bochner-Riesz means  for critical magnetic Schr\"odinger operators in ${\mathbb R^2}$}

\author{Changxing Miao}
\address{School of Mathematics and Physics, University of Science and Technology Beijing, Beijing 100083, P. R.
CHINA}
\email{miao\_changxing@ustb.edu.cn}

\author{Lixin Yan}
\address{Department of Mathematics, Sun Yat-Sen University, Guangzhou, 510275, P.R.
CHINA}
\email{mcsylx@mail.sysu.edu.cn}

\author{Junyong Zhang}
\address{Department of Mathematics, Beijing Institute of Technology, Beijing 100081, P.R.
CHINA; }
\email{zhang\_junyong@bit.edu.cn}

\begin{abstract}
We study $L^p$-boundedness of the  Bochner-Riesz means
for critical magnetic Schr\"odinger operators $\LL_{\A}$ in ${\mathbb R^2}$, which involve the {physical} Aharonov-Bohm potential.   We show that
for $1\leq p\leq +\infty$ and $p\not= 2$, 
the Bochner-Riesz operator   $S_{\lambda}^\delta(\LL_{\A})$ of order $\delta$ is bounded on $L^p(\R^2)$ if and only if  $\delta>\max\big\{0, 2\big|1/2-1/p\big|-1/2\big\}$. The new ingredient {in} the proof is
to obtain the localized $L^4(\R^2)$ estimate of $S_{\lambda}^\delta(\LL_{\A})$, whose kernel  is heavily affected by the physical magnetic diffraction,   and more singular than  the classical Bochner-Riesz  means $S_{\lambda}^\delta(\Delta)$ for the   Laplacian $\Delta$ in ${\mathbb R}^2$.

\end{abstract}


\maketitle 

\section{Introduction}
Consider a scaling invariant magnetic Schr\"odinger operator of the form
\begin{equation}\label{LA}
	\mathcal{L}_{{\A}}=-\left(\nabla+i\frac{{\A}(\hat{x})}{|x|}\right)^2,\quad x\in \R^2\setminus\{0\},
\end{equation}
where $\hat{x}=\tfrac{x}{|x|}\in\mathbb{S}^1$ and ${\A}\in W^{1,\infty}(\mathbb{S}^1;\R^2)$  satisfies the transversality condition
\begin{equation}\label{eq:transversal}
	{\A}(\hat{x})\cdot\hat{x}=0,
	\qquad
	\text{for all }x\in\R^2.
\end{equation}
 A primary example is 
the {\it Aharonov-Bohm} potential
\begin{equation}\label{ab-potential}
	\frac{{\A}(\hat{x})}{|x|}=\alpha\left(-\frac{x_2}{|x|^2},\frac{x_1}{|x|^2}\right),\quad \alpha\in\R,
\end{equation}
which was introduced in \cite{AB}, in the context of Schr\"odinger dynamics, to show that scattering effects can even occur in regions in which the electromagnetic field is absent.
It is well known that the Aharonov-Bohm (AB) effect  plays an important role in quantum theory refining the status of electromagnetic potentials (see \cite{AB}). 
This effect describes the interaction between a non-relativistic charged particle and an infinitely long and infinitesimally thin magnetic solenoid field (hereafter the AB field). We refer to Avron-Herbst-Simon \cite{AHS1,AHS2, AHS3} and Reed-Simon \cite{RS} for electromagnetic {phenomena in} quantum mechanics,
in which many important physical potentials (e.g., the constant magnetic field and the Coulomb electric potential) are discussed. The magnetic potential is scaling critical, making the perturbation completely nontrivial. From a mathematical perspective,   the singularity of the Aharonov-Bohm potential  is not locally square integrable (i.e., not in $L^2_{\text{loc}}(\R^2)$),  which makes the analysis of this operator more difficult compared to the usual theory for magnetic Schr\"odinger Hamiltonians (see \cite{RS04}). {On the other hand}, the scaling invariance is advantageous from the perspective of Fourier analysis and allows for explicit calculations of generalized eigenfunctions, as we shall see in the sequel.\vspace{0.1cm}

Recently, there has been growing interest in the harmonic analysis associated with the above conical singular operators. In \cite{FZZ, FZZ1}, 
  Fanelli, Zheng and the   third named author  proved Strichartz estimates for the wave equation and 
the $L^p-L^q$ resolvent estimates; and later in \cite{GYZZ}, together with Gao, Yin and Zheng,  the   third named author   {proved} Strichartz estimates for the Klein-Gordon equation by constructing the spectral measure of the operator $\mathcal{L}_{{\A}}$. We 
studied $L^2$-weighted resolvent estimates in \cite{GWZZ}. Interesting physical phenomena such as diffraction and resonance are {known to play an important role} in the harmonic analysis theory associated with the above conical singular operators.  \vspace{0.1cm}

In the present paper, we focus on purely magnetic potentials that are critical with respect to the scaling of the free operator. Many problems in harmonic analysis (e.g., resolvent estimates, Strichartz estimates) {for} the Schr\"odinger operators with critical potentials become harder, since the usual perturbation arguments are not available. We refer to 
\cite{BVZ, CK16} for the resolvent estimates,  in which the potential is required to decay faster at infinity than the critical one. We refer to \cite{CS,DF, DFVV, EGS1, EGS2, S} and the references therein for time-decay or Strichartz estimates with subcritical magnetic potentials. However, the problems with critical magnetic potentials are more delicate and even remain open in high dimension.  The space dimension $n=2$ is very peculiar for this kind of problem, since the associated spherical problem is one-dimensional and several explicit expansions are available, leading to quite complete results.  For example, the authors in \cite{FFFP1,FFFP} proved the time-decay estimate
for the Schr\"odinger equation; while  \cite{FZZ, GYZZ} established the Strichartz estimates for the wave and Klein-Gordon equations. \vspace{0.2cm}

The purpose of this paper is to investigate 
  the Bochner-Riesz operator  $S_{\lambda}^\delta(\LL_{\A}): = (1- {\LL_{\A}}/{\lambda^2} )^\delta_{+}$ of order $\delta$ associated with a scaling critical magnetic Schr\"odinger operator $\mathcal L_{\A}$ in ${\mathbb R^2}$, which is defined, for $\delta\geq 0$ and $\lambda>0$, by
\begin{eqnarray}\label{oper:BR}
S_{\lambda}^\delta(\LL_{\A})
=\int_0^{\lambda} \Big(1-\frac{\rho^2}{\lambda^2}\Big)^\delta  dE_{\sqrt{\LL_{{\A}}}} (\rho), 
\end{eqnarray}
where $E_{\sqrt{\LL_{{\A}}}} (\rho) $ is a spectral resolution of $\sqrt{\LL_{{\A}}}$. By the spectral theorem,  the operator $S_{\lambda}^\delta(\LL_{\A})$ is bounded on $L^2({\mathbb R^2})$.
The problem known as the Bochner-Riesz problem is to determine the optimal index $\delta$ for $1\leq p\leq +\infty$ such that 
$S_{\lambda}^\delta(\LL_{\A})f$ converges to $f$ in $L^p({\mathbb R^2})$ for every $f\in L^p({\mathbb R^2})$.
This kind of problem was initially considered for the Laplacian $ \Delta=-\sum_{i=1}^n \partial^2_{x_i}$
on the Euclidean space $\R^n$ with $n\geq 2$, and the problem has been extensively studied by numerous authors.
It has been  conjectured    that the classical Bochner–Riesz mean $S_{\lambda}^\delta(\Delta)f$ converges
in $L^p(\R^n)$ if and only if
\begin{equation}\label{delta}
\delta>\delta_c(p,n):=\max\left\{0, n\big|\frac12-\frac1p\big|-\frac12\right\}
\end{equation}
for $ p\in [1,+\infty]\setminus\{2\}$.
When $p = 2$, the convergence holds true if and only if $\delta\geq 0$ by
Plancherel’s theorem. In two dimensions, the conjecture has been proved by L. Carleson and P. Sj\"olin \cite{CS} by means of a more general theorem on oscillatory integral operators.  For a variant of their proof, we refer to L. H\"ormander \cite{Hor},  C. Fefferman \cite{Fef2}  and A. Cordoba \cite{Cor}. However, in higher dimensions, it remains open and partial results are known. For readers who are interested in recent progress on the conjecture, we refer to \cite{Bour, BG, GOWWZ, Lee, Stein, Ta3} and references therein.
\vspace{0.1cm}

The Bochner-Riesz   means  for elliptic operators   (e.g., Schr\"odinger operators with potentials, the Laplacian on manifolds) have attracted a lot of attention and {have been extensively studied by many authors}, see
 \cite{AW, COSY, CSo, DOS, GHS, K, KST, H2, JLR1, JLR, LR, Sogge87, T2}. In particular, we refer to Sogge \cite{Sogge87}, Christ-Sogge \cite{CSo} for the Bochner-Riesz means on compact manifolds;
 Guillarmou-Hassell-Sikora \cite{GHS} for the Bochner-Riesz means on non-trapping asymptotically conic manifolds;  Askey-Wainger \cite{AW}, Karadzhov \cite{K},   Thangavelu \cite{T2} and    
   Lee-Ryu \cite{LR}  for the Bochner-Riesz means   for the Hermite operator, and   Jeong-Lee-Ryu \cite{JLR} for the Bochner-Riesz means   for the twisted Laplacian in $\R^2$. 
However, the picture of the Bochner-Riesz  problem associated with elliptic operators is far from complete. Motivated by this, we
study the Bochner-Riesz problem for variable-coefficient Schr\"odinger operators with a scaling-critical magnetic potential. \vspace{0.1cm}

Our main result is the following.

\begin{theorem}\label{thm:LA0} Let $1\leq p\leq +\infty$ and let 
$S_{\lambda}^\delta(\LL_{\A})$ be the Bochner-Riesz means defined by
\eqref{LA} and \eqref{oper:BR}.
Then, there exists  a constant $C$ independent of $\lambda>0$ such that 
\begin{equation}\label{est:BR}
\|S_{\lambda}^\delta(\LL_{\A})f\|_{L^p(\R^2)}\leq C\|f\|_{L^p(\R^2)},
\end{equation}
if and only if  $\delta>\delta_c(p,2)=\max\big\{0, 2\big|1/2-1/p\big|-1/2\big\}$ when $p\not=2$.
\end{theorem}

The necessary part of \eqref{est:BR} follows from a transplantation theorem due to Kenig-Stanton-Tomas \cite{KST} and the necessary condition for the $L^p$-boundedness of the Bochner-Riesz means for $S_{\lambda}^\delta(\Delta)$ when the total magnetic flux $\alpha\in\Z$.
If $p=2$, one can use the spectral theorem to prove that \eqref{est:BR} holds for $\delta\geq 0$. 

To show the sufficiency part of \eqref{est:BR}, 
we note that the magnetic potential considered here is singular at the origin and decays like the Coulomb potential at infinity, which is quite different from the harmonic oscillator potential in \cite{LR} and the constant magnetic field in \cite{JLR}. The potentials considered in their papers belong to $L^2_{\text{loc}}$ but are unbounded at infinity, so that these Schr\"odinger operators have only pure point spectrum and have a nice
representation of the kernel due to the Mehler formula.  Our magnetic potential is too singular at the origin to be in $L^2_{\text{loc}}$, while its decay at infinity is not sufficiently fast, leading to a difficult long-range perturbation.
The {sufficiency part} of \eqref{est:BR} can be proved by making use of the ideas of the celebrated oscillatory integral theory of Stein \cite{Stein} and H\"ormander \cite{Hor} adapted to the singular kernel, which is the so-called variable-coefficient Plancherel theorem of T. H. Wolff \cite{Wolff}. But we have to overcome the difficulties caused by the singularity of the potential and the Aharonov-Bohm effect of the operator $\LL_{\A}$.
Let us now describe some key points of our proof.\vspace{0.2cm}

\begin{itemize}

\item The magnetic Aharonov-Bohm effect is a quantum mechanical phenomenon in which
an electrically charged particle is affected by an electromagnetic potential $\frac{{\A}(\hat{x})}{|x|}$ in \eqref{ab-potential}, despite being
confined to a region in which the magnetic field $${\bf B}(x)=\nabla \times \frac{{\A}(\hat{x})}{|x|}=2\pi \alpha\delta_0(x)$$ vanishes. In physics, the Hamiltonian \eqref{LA} {describes} an infinitely long solenoid {of} radius $r\to 0$.  Assume the total magnetic flux $\alpha\notin \Z$. {The motion of a classical particle} outside the solenoid
will not be affected by the vector potential due to the {absence of the} magnetic field. However, in
quantum mechanics, consider a charged particle traveling around the solenoid; it picks up a
non-trivial phase shift if it travels around the solenoid once. This phenomenon shows the quantum significance of
the vector potential.\vspace{0.1cm}

\item To understand the Hamiltonian $\LL_{\A}$ on $\R^2\setminus\{0\}$, one needs {the} diffractive geometry of the conical singular operator of Ford-Wunsch \cite{FW}. By
standard propagation of singularities, away from the solenoid set $\{0\}$, singularities propagate along straight lines in $\R^2\setminus\{0\}$. However, near the solenoid there
are two other types of generalized geodesics, along which the singularities propagate, passing
through the solenoid, which correspond to the diffractive and geometric waves emanating
from the solenoid upon diffraction. Therefore, in our proof, the difference between the incoming and outgoing directions plays an important role. In particular, when the difference $\theta_1-\theta_2$ is near $\pi$, the proof becomes more delicate.  \vspace{0.1cm}

\item In contrast to Yang \cite{Yang21,Yang22, Yang23}, in which diffractive geometry and the estimates in weighted $L^2$ spaces are studied, the $L^p$-framework estimates on Bochner-Riesz means require
the elaborate properties of the kernel which captures the decay and the oscillation behavior.  \vspace{0.1cm}

\item {Mathematically,} the incident wave is different from the usual plane wave $e^{ix\cdot\xi}$, which leads to the appearance of the modified factor $e^{\pm i\alpha(\theta_1-\theta_2)}$ in the kernel of Bochner-Riesz means because of the long-range property of the potential $\frac{{\A}(\hat{x})}{|x|}$. In addition, the Heaviside step function $\mathbbm{1}_{I}(\theta_1-\theta_2)$ appears in the kernel of Bochner-Riesz means,
which is another obstacle to efficiently exploit the oscillation behavior of the kernel.  \vspace{0.1cm}

\item In contrast to our previous resolvent estimates in \cite{FZZ1}, we have to study a more difficult kernel 
and prove more localized estimates. To prove the sharp result in Theorem \ref{thm:LA0}, we first prove the localized estimates on $L^p(\R^2)\to L^{p}(\R^2)$ with $p\geq 6$;
and then prove the localized estimate on $L^4(\R^2)\to L^{4}(\R^2)$, which is the key estimate on the line $q=3p'$, see \cite{Sogge1, Stein}.
\vspace{0.1cm}

\end{itemize}

The paper is organized as follows. In Section \ref{sec:pre}, we sketch the construction of the resolvent and spectral measure kernels of the operator $\LL_{{\A}}$  and we further construct the kernel of the Bochner-Riesz means operator. In Section \ref{sec:thmmain}, we prove  the sufficiency part of \eqref{est:BR}  by using two key localized propositions, namely Proposition \ref{prop:TGj} and Proposition \ref{prop:TDj} below, which will be shown in  Section \ref{sec:TGj} and Section \ref{sec:TDj}, respectively. Finally, in the appendix, we prove the technical lemmas used  in Section \ref{sec:TGj} and Section \ref{sec:TDj}.
\vspace{0.2cm}

{\bf Acknowledgement.} 
This project was supported by the National Key R\&D Program of
China: 2022YFA1005700, National Natural
Science Foundation of China (12171031, 12531005, 12571111) and Beijing Natural Science Foundation (1242011).




\section{The kernel of Bochner-Riesz means}\label{sec:pre}

In this section, we first sketch the construction of the resolvent and spectral measure kernels for the Schrödinger operator $\LL_{{\A}}$, which is more convenient for studying the Bochner–Riesz means. We then construct the kernel of the Bochner–Riesz means $S_{\lambda}^\delta(\LL_{\A})$, which captures the decay and oscillation properties.

\subsection{The resolvent and spectral measure kernels}
In this subsection, we sketch the construction of
the resolvent and the spectral measure kernels for the 
Schr\"odinger operator $\LL_{{\A}}$, which have been studied in connection with the wave and Klein-Gordon equations in \cite{FZZ, GYZZ}.
For the sake of self-containedness, we outline the main steps of the proof and provide some physical diffractive interpretations of the kernel representations.

By using the conic structure of $\mathcal{L}_{{\A}}$ in \eqref{LA}, in {polar coordinates} we can write
\begin{equation}\label{LA-r}
\begin{split}
\mathcal{L}_{{\A}}=-\partial_r^2-\frac{1}r\partial_r+\frac{L_{{\A}}}{r^2},
\end{split}
\end{equation}
where the operator
\begin{equation}\label{L-angle}
\begin{split}
L_{{\A}}&=(i\nabla_{\mathbb{S}^{1}}+{\A}(\hat{x}))^2,\qquad \hat{x}\in \mathbb{S}^1
\\&=-\Delta_{\mathbb{S}^{1}}+\big(|{\A}(\hat{x})|^2+i\,\mathrm{div}_{\mathbb{S}^{1}}{\A}(\hat{x})\big)+2i {\A}(\hat{x})\cdot\nabla_{\mathbb{S}^{1}}.
\end{split}
\end{equation}
Let $\hat{x}=(\cos\theta,\sin\theta)$, then
\begin{equation*}
\partial_\theta=-\hat{x}_2\partial_{\hat{x}_1}+\hat{x}_1\partial_{\hat{x}_2},\quad \partial_\theta^2=\Delta_{\mathbb{S}^{1}}.
\end{equation*}
Define $A(\theta):[0,2\pi)\to \R$ by
\begin{equation}\label{equ:alpha}
A(\theta)={\bf A}(\cos\theta,\sin\theta)\cdot (-\sin\theta,\cos\theta).
\end{equation}
Then, using the transversality condition \eqref{eq:transversal}, we can write
\begin{equation*}
{\bf A}(\cos\theta,\sin\theta)=A(\theta)(-\sin\theta,\cos\theta),\quad \theta\in[0,2\pi).
\end{equation*}
Thus, we obtain
\begin{equation}\label{LAa-s}
\begin{split}
L_{{\A}}&=-\Delta_{\mathbb{S}^{1}}+\big(|{\A}(\hat{x})|^2+i\,\mathrm{div}_{\mathbb{S}^{1}}{\A}(\hat{x})\big)+2i {\A}(\hat{x})\cdot\nabla_{\mathbb{S}^{1}}\\
&=-\partial_\theta^2+\big(|A(\theta)|^2+i\,{A'}(\theta)\big)+2i A(\theta)\partial_\theta\\
&=(i\partial_\theta+A(\theta))^2.
\end{split}
\end{equation}
For simplicity, we define the constant $\alpha$ by
$$\alpha=\Phi_{\A}=\frac1{2\pi}\int_0^{2\pi} A(\theta) \, d \theta,$$
which represents the total magnetic flux.\vspace{0.2cm}

For $\lambda>0$, we define the resolvent of the self-adjoint operator $\LL_{\A}$ by
\begin{equation}\label{def:res}
\big(\LL_{{\A}}-(\lambda^2\pm i0)\big)^{-1}
=\lim_{\epsilon\searrow0}\big(\LL_{{\A}}-(\lambda^2\pm i\epsilon)\big)^{-1},
\end{equation}
where we use the same notation $\LL_{\A}$ {for the} Friedrichs self-adjoint extension {of} the Hamiltonian \eqref{LA}.\vspace{0.2cm}

In \cite{FZZ, GYZZ}, we observe the conic property of this operator, which allows us to construct the resolvent kernel
inspired by Cheeger-Taylor \cite{CT1,CT2}, where the authors studied the diffractive waves on flat cones.

\begin{proposition}[Resolvent kernel,  \cite{GYZZ}]\label{prop:res-ker}
Let $x=r_1(\cos\theta_1,\sin\theta_1)$ and $y=r_2(\cos\theta_2,\sin\theta_2)$,
then the resolvent kernel has the following expression
\begin{align}\label{equ:res-ker-out}
     \big(\LL_{{\A}}-(\lambda^2\pm i0)\big)^{-1}(x,y)=&\frac1{\pi}\int_{\R^2}\frac{e^{-i(x-y)\cdot\xi}}{|\xi|^2-(\lambda^2\pm i0)}\;d\xi\, A_\alpha(\theta_1,\theta_2)\\\nonumber
  &+\frac1{\pi}\int_0^\infty \int_{\R^2}\frac{e^{-i{\bf n}_s\cdot\xi}}{|\xi|^2-(\lambda^2\pm i0)}\;d\xi \, B_\alpha(s,\theta_1,\theta_2)\;ds,
\end{align}
where  ${\bf n}_s=(r_1+r_2, \sqrt{2r_1r_2(\cosh s-1)})$
and where
\begin{equation}\label{A-al}
\begin{split}
&A_{\alpha}(\theta_1,\theta_2)= \frac{e^{i\int_{\theta_1}^{\theta_2}\,A(\theta')d\theta'}}{4\pi^2}\\
&\times \big(\mathbbm{1}_{[0,\pi]}(|\theta_1-\theta_2|)
  +e^{-i2\pi\alpha}\mathbbm{1}_{[\pi,2\pi)}(\theta_1-\theta_2)+e^{i2\pi\alpha}\mathbbm{1}_{(-2\pi,-\pi]}(\theta_1-\theta_2)\big),
  \end{split}
  \end{equation}
 and
\begin{equation}\label{B-al}
\begin{split}
&B_{\alpha}(s,\theta_1,\theta_2)= -\frac{1}{4\pi^2}e^{-i\alpha(\theta_1-\theta_2)+i\int_{\theta_2}^{\theta_{1}} A(\theta') d\theta'}  \Big(\sin(|\alpha|\pi)e^{-|\alpha|s}\\
&\qquad +\sin(\alpha\pi)\frac{(e^{-s}-\cos(\theta_1-\theta_2+\pi))\sinh(\alpha s)-i\sin(\theta_1-\theta_2+\pi)\cosh(\alpha s)}{\cosh(s)-\cos(\theta_1-\theta_2+\pi)}\Big).
 \end{split}
\end{equation}

\end{proposition}

\begin{remark} From \eqref{equ:res-ker-out}, we see that {the resolvent kernel} consists of two terms.
The first one is close to the resolvent kernel of $-\Delta$ but with $A_{\alpha}(\theta_1,\theta_2)$ due to the long-range perturbation of the potential $\frac{{\A}(\hat{x})}{|x|}$.
The second one arises from the diffractive effect, and $B_{\alpha}(s,\theta_1,\theta_2)$ is singular when $\theta_1-\theta_2\to \pi$ and $s\to 0$. As explained above, both terms appear
in the kernel of {the} Bochner-Riesz means.
\end{remark}

\begin{remark} In particular, if the total magnetic flux satisfies $\alpha\in\Z$, {then} $B_{\alpha}(s,\theta_1,\theta_2)$ vanishes and $A_{\alpha}(\theta_1,\theta_2)$ becomes a constant.
The resolvent {is consistent with} the resolvent kernel of the Euclidean Laplacian without magnetic potential. This corresponds to the physical fact that magnetic Schr\"odinger operators are unitarily equivalent.
\end{remark}

\begin{remark} The scaling-critical magnetic Schr\"odinger operators \eqref{LA} are essentially equivalent to the Aharonov–Bohm Schr\"odinger operator
after a gauge transformation. Specifically, from \eqref{A-al} and \eqref{B-al}, one can observe the kernel identity 
$$ {e^{-i\phi(\theta_1)}}\big(\LL_{{\A}}-(\lambda^2\pm i0)\big)^{-1}(x,y) {e^{i\phi(\theta_2)}}= \big(\LL_{{\alpha}}-(\lambda^2\pm i0)\big)^{-1}(x,y)$$
where $\phi(\theta)=\int_0^\theta(A(\theta')-\alpha)\, d\theta'$ and the Aharonov–Bohm Schr\"odinger operator $\LL_{{\alpha}}=\LL_{{\A}}$ with  ${\bf A}(\hat{x})$ given in \eqref{ab-potential}.
\end{remark}

\begin{remark} Here we refine the $A_{\alpha}(\theta_1,\theta_2)$ of \cite{GYZZ} by
\begin{equation*}
A_{\alpha}(\theta_1,\theta_2)= \frac{e^{i\int_{\theta_1}^{\theta_2} A(\theta')d\theta'}}{4\pi^2}\big(\mathbbm{1}_{[0,\pi]}(|\theta_1-\theta_2|)
  +e^{-i2\pi\alpha}\mathbbm{1}_{[\pi,2\pi)}(\theta_1-\theta_2)+e^{i2\pi\alpha}\mathbbm{1}_{(-2\pi,-\pi]}(\theta_1-\theta_2)\big).
  \end{equation*}
Since $\theta_1,\theta_2\in [0, 2\pi)$, we have $\theta_1-\theta_2\in (-2\pi, 2\pi)$, and thus we obtain
\begin{align*}
&\sum_{\{j\in\Z: 0\leq |\theta_1-\theta_2+2j\pi|\leq \pi\}} e^{z\cos(\theta_1-\theta_2+2j\pi)}e^{i(\theta_1-\theta_2+2j\pi)\alpha}
\\
 =&\frac1{2\pi}\times\begin{cases}
e^{z\cos(\theta_1-\theta_2)} e^{i(\theta_1-\theta_2)\alpha}\quad&\text{if}\quad |\theta_1-\theta_2|<\pi\\
e^{z\cos(\theta_1-\theta_2)}e^{i(\theta_1-\theta_2-2\pi)\alpha}\quad&\text{if}\quad \pi<\theta_1-\theta_2<2\pi\\
e^{z\cos(\theta_1-\theta_2)}e^{i(\theta_1-\theta_2+2\pi)\alpha}\quad&\text{if}\quad -2\pi<\theta_1-\theta_2<-\pi\\
e^{-z}\big(e^{i\pi\alpha}+e^{-i\pi\alpha}\big)\quad&\text{if}\quad |\theta_1-\theta_2|=\pi.
 \end{cases}
\end{align*}

\end{remark}

\begin{proof}[Proof of Proposition \ref{prop:res-ker}] We outline the proof here for completeness and we refer to \cite{GYZZ, FZZ} for details.	
The outgoing resolvent can be written as
\begin{equation}\label{equ:resfor}
  \big({\LL_{{\A}}}-(\lambda^2+i0)\big)^{-1}=-\frac1i\lim_{\epsilon\searrow0}\int_0^\infty e^{-it{\LL_{{\A}}}}e^{it(\lambda^2+i\epsilon)}\;dt,
\end{equation}
by observing that for $z=\lambda^2+i\epsilon\in\{z\in\C:\;{\rm Im}(z)>0\}$
\begin{equation*}
  (s-z)^{-1}=-\frac1i\int_0^\infty e^{-ist}e^{izt}\;dt,\quad \forall\;s\in\R.
\end{equation*}
We construct the Schr\"odinger propagator 
  \begin{align}\label{equ:ktxyschr}
    e^{-it\LL_{{\A}}}(x,y)=&\sum_{k\in\Z}\varphi_{k}(\theta_1)\overline{\varphi_{k}(\theta_2)}K_{\nu_k}(t; r_1,r_2)\\\nonumber
    =&\frac1{2\pi}e^{-i\alpha(\theta_1-\theta_2)+i\int_{\theta_2}^{\theta_{1}}A(\theta') d\theta'} \sum_{k\in\Z} e^{-ik(\theta_1-\theta_2)}  K_{\nu_k}(t; r_1,r_2),
  \end{align}
  by making use of the spectral argument and the functional calculus. Here
  $\nu_k=|k+\alpha|, k\in\Z$ is the square root of the eigenvalue of the operator $L_{{\A}}$ in \eqref{LAa-s} and 
\begin{equation}\label{eigf}
\varphi_k(\theta)=\frac1{\sqrt{2\pi}}e^{-i\big(\theta(k+\alpha)-\int_0^{\theta}A(\theta') d\theta'\big)}
\end{equation}
is the corresponding eigenfunction.
The radial kernel $K_{\nu}$ is given by
\begin{align}\label{equ:knukdef12sch}
  K_{\nu}(t,r_1,r_2)=&\int_0^\infty e^{-it\rho^2}J_{\nu}(r_1\rho)J_{\nu}(r_2\rho) \,\rho d\rho\\\nonumber
  =&\lim_{\epsilon\searrow0}\int_0^\infty e^{-(\epsilon+it)\rho^2}J_{\nu}(r_1\rho)J_{\nu}(r_2\rho) \,\rho d\rho\\\nonumber
  =&\lim_{\epsilon\searrow0}\frac{e^{-\frac{r_1^2+r_2^2}{4(\epsilon+it)}}}{2(\epsilon+it)} I_\nu\Big(\frac{r_1r_2}{2(\epsilon+it)}\Big),
\end{align}
where we use the Weber identity \cite[Proposition 8.7]{Taylor} and $I_\nu$ is the modified Bessel function
\begin{equation}\label{m-bessel}
I_\nu(z)=\frac1{\pi}\int_0^\pi e^{z\cos s} \cos(\nu s) ds-\frac{\sin(\nu\pi)}{\pi}\int_0^\infty e^{-z\cosh s} e^{-s\nu} ds.
\end{equation}
Using the Poisson summation formula, we can sum in $k$ to obtain
the kernel of {the} Schr\"odinger propagator
\begin{equation}\label{equ:Skernel}
\begin{split}
  &e^{-it\LL_{{\A}}}(x,y)
  =\frac{e^{-\frac{|x-y|^2}{4it}} }{it}A_{\alpha}(\theta_1,\theta_2)+\frac{e^{-\frac{r_1^2+r_2^2}{4it}} }{it}  \int_0^\infty e^{-\frac{r_1r_2}{2it}\cosh s} B_{\alpha}(s,\theta_1,\theta_2) ds.
\end{split}
\end{equation}
For $z=\lambda^2+i\epsilon$ with $\epsilon>0$, we observe that
\begin{align*}
 & \int_0^\infty \frac{e^{-\frac{|x-y|^2}{4it}} }{it}e^{itz}\;dt
  =\frac{1}{i\pi}\int_{\R^2}\frac{e^{-i(x-y)\cdot\xi}}{|\xi|^2-z}\;d\xi,
\end{align*}
and
\begin{equation}\label{equ:intr1r223}
  \int_0^\infty  \frac{e^{-\frac{r_1^2+r_2^2}{4it}} }{it} e^{-\frac{r_1r_2}{2it}\cosh s} e^{itz}\;dt=\frac{1}{i\pi}\int_{\R^2}\frac{e^{-i{\bf n}_s\cdot\xi}}{|\xi|^2-z}\;d\xi,
\end{equation}
where ${\bf n}_s=(r_1+r_2, \sqrt{2r_1r_2(\cosh s-1)})$. Substituting \eqref{equ:Skernel} into \eqref{equ:resfor} yields the outgoing resolvent of \eqref{equ:res-ker-out}.
The incoming resolvent of \eqref{equ:res-ker-out} can be proved 
by noting 
\begin{equation}\label{out-inc}
\begin{split}
(\mathcal L_{{\A}}-(\lambda^2-i0))^{-1}=\overline{(\overline{\mathcal L_{{\A}}}-(\lambda^2+i0))^{-1}}.
\end{split}
\end{equation}
Indeed, from \eqref{LA-r} and \eqref{LAa-s}, $\overline{\mathcal L_{{\A}}}$ coincides with $\mathcal L_{{\A}}$ after replacing $A(\theta)$ by $-A(\theta)$.
Together with the facts that $\overline{A_{-\alpha}} = A_{\alpha}$ and $\overline{B_{-\alpha}} = B_{\alpha}$, this yields the incoming resolvent expressed in \eqref{equ:res-ker-out}.
\end{proof}

According to Stone's formula, the spectral measure is related to the resolvent
\begin{equation}\label{equ:spemes}
   dE_{\sqrt{\LL_{\A}}}(\lambda)=\frac{d}{d\lambda}E_{\sqrt{\LL_{\A}}}(\lambda)\;d\lambda
   =\frac{\lambda}{i\pi}\big(R(\lambda+i0)-R(\lambda-i0)\big)\;d\lambda,
\end{equation}
where the resolvent is given by
$$R(\lambda\pm i0)=\big(\LL_{{\A}}-(\lambda^2\pm i0)\big)^{-1}=\lim_{\epsilon\searrow0}\big(\LL_{{\A}}-(\lambda^2\pm i\epsilon)\big)^{-1}.$$
\begin{proposition}\label{prop:spect}
Let $x=r_1(\cos\theta_1, \sin\theta_1)$ and $y=r_2(\cos\theta_2, \sin\theta_2)$ in $\R^2\setminus\{0\}$. Then the Schwartz kernel of the spectral measure satisfies
 \begin{equation}\label{ker:spect}
 \begin{split}
 dE_{\sqrt{\LL_{{\A}}}}(\lambda;x,y) =&
\frac{\lambda}{\pi} \Big(
 \int_{\mathbb{S}^1} e^{-i\lambda (x-y)\cdot\omega} d\sigma_\omega A_{\alpha}(\theta_1,\theta_2)
 \\&\quad
+\int_0^\infty \int_{\mathbb{S}^1} e^{-i\lambda {\bf n}_s\cdot\omega} d\sigma_\omega
B_{\alpha}(s,\theta_1,\theta_2) ds\Big),
\end{split}
\end{equation}
where $A_{\alpha}(\theta_1,\theta_2)$ and $B_{\alpha}(s,\theta_1,\theta_2)$ are given in \eqref{A-al} and \eqref{B-al}.

\end{proposition}

\begin{remark}
The spectral measure kernel  differs slightly from the one stated in \cite{GYZZ}. Here we retain the formula involving the Fourier transform of the measure on the circle $\mathbb{S}^1$, as it is more convenient for computing the Bochner-Riesz kernel. 
\end{remark}

\begin{proof}[Proof of Proposition \ref{prop:spect}] We just modify the proof in \cite{GYZZ} to obtain \eqref{ker:spect}, which is based on the resolvents in Proposition \ref{prop:res-ker}.
Then, by Stone's formula \eqref{equ:spemes} and Proposition \ref{prop:res-ker}, we obtain
\begin{equation}\label{eq:spect}
\begin{split}
  &dE_{\sqrt{\LL_{\A}}}(\lambda; x,y)\\
  = &\frac1{\pi} \frac{\lambda}{i\pi} A_\alpha(\theta_1,\theta_2)\int_{\R^2} e^{-i(x-y)\cdot\xi}\Big(\frac{1}{|\xi|^2-(\lambda^2+i0)}-\frac{1}{|\xi|^2-(\lambda^2-i0)}\Big)\;d\xi\\
  &+\frac1{\pi} \frac{\lambda}{i\pi} \int_0^\infty \Big[ \int_{\R^2} e^{-i{\bf n}_s\cdot\xi}\Big(\frac{1}{|\xi|^2-(\lambda^2+i0)}-\frac{1}{|\xi|^2-(\lambda^2-i0)}\Big)\;d\xi \Big]B_\alpha(s,\theta_1,\theta_2)\;ds.
  \end{split}
\end{equation}
We observe that
\begin{equation}\label{id-spect}
\begin{split}
&\lim_{\epsilon\to 0^+}\frac{\lambda}{i\pi }\int_{\R^2} e^{-ix\cdot\xi}\Big(\frac{1}{|\xi|^2-(\lambda^2+i\epsilon)}-\frac{1}{|\xi|^2-(\lambda^2-i\epsilon)}\Big) d\xi\\
=&\lim_{\epsilon\to 0^+} \frac{\lambda}{\pi }\int_{\R^2} e^{-ix\cdot\xi}\Im\Big(\frac{1}{|\xi|^2-(\lambda^2+i\epsilon)}\Big)d\xi\\
=&\lim_{\epsilon\to 0^+} \frac{\lambda}{\pi }\int_{0}^\infty \frac{\epsilon}{(\rho^2-\lambda^2)^2+\epsilon^2} \int_{|\omega|=1} e^{-i\rho x\cdot\omega} d\sigma_\omega  \, \rho d\rho\\
=& \lambda \int_{|\omega|=1} e^{-i\lambda x\cdot\omega} d\sigma_\omega ,
\end{split}
\end{equation}
where we use the fact that
 the  Poisson kernel is an approximation to the identity which means that, for any reasonable function $m(x)$
\begin{equation}
\begin{split}
m(x)&=\lim_{\epsilon\to 0^+}\frac1\pi \int_{\R} {\rm Im}\Big(\frac{1}{x-(y+i\epsilon)}\Big) m(y)dy
\\&=\lim_{\epsilon\to 0^+}\frac1\pi \int_{\R} \frac{\epsilon}{(x-y)^2+\epsilon^2} m(y)dy.
\end{split}
\end{equation}
Therefore we plug \eqref{id-spect} into \eqref{eq:spect} to obtain that
 \begin{equation*}
 \begin{split}
 dE_{\sqrt{\LL_{{\A}}}}(\lambda;x,y) =&
\frac{\lambda}{\pi} \Big(
 \int_{\mathbb{S}^1} e^{-i\lambda (x-y)\cdot\omega} d\sigma_\omega A_{\alpha}(\theta_1,\theta_2)
 \\&\quad
+\int_0^\infty \int_{\mathbb{S}^1} e^{-i\lambda {\bf n}_s\cdot\omega} d\sigma_\omega
B_{\alpha}(s,\theta_1,\theta_2) ds\Big).
\end{split}
\end{equation*}
Hence, we prove Proposition \ref{prop:spect}.
\end{proof}

\subsection{The kernel of Bochner-Riesz means }

Let us recall that the Bochner-Riesz means of order $\delta$  are defined by the formula
\begin{equation}
S_{\lambda}^\delta(\LL_{\A})=\Big(1-\frac{\LL_{\A}}{\lambda^2}\Big)^\delta_{+},\quad \lambda>0.
\end{equation}
Here we use the distribution notation $\chi_+^a(x)$, defined by $\chi_+^a(x)=x_+^a/\Gamma(a+1)$, where $\Gamma$ is the gamma function and
\begin{equation}
x_+^a=
\begin{cases} x^a, \quad &\text{if}\, x\geq 0,\\
0, \quad &\text{if}\, x< 0.
\end{cases}
\end{equation}
By using Proposition \ref{prop:spect}, {the kernel can be represented as}
\begin{equation}\label{ker:BR}
\begin{split}
\Big(1-\frac{\LL_{\A}}{\lambda^2}\Big)^\delta_{+}(x,y)&=\int_0^\infty \Big(1-\frac{\rho^2}{\lambda^2}\Big)^\delta_{+} dE_{\sqrt{\LL_{{\A}}}}(\rho;x,y)  \, d\rho\\
&=\frac1{\pi}\int_0^\infty \rho \Big(1-\frac{\rho^2}{\lambda^2}\Big)^\delta_{+} \left(
 \int_{\mathbb{S}^1} e^{-i\rho (x-y)\cdot\omega} d\sigma_\omega A_{\alpha}(\theta_1,\theta_2)\right.
 \\&\quad \left.
+\int_0^\infty \int_{\mathbb{S}^1} e^{-i\rho {\bf n}_s\cdot\omega} d\sigma_\omega
\, B_{\alpha}(s,\theta_1,\theta_2) ds\right)\, d\rho. \end{split}
\end{equation}
More precisely, the kernel representation of the Bochner-Riesz means is given by 
\begin{proposition}\label{prop: ker-BR} Let $x=r_1(\cos\theta_1,\sin\theta_1)$ and $y=r_2(\cos\theta_2,\sin\theta_2)$.
Define
\begin{equation}\label{d-j}
d(r_1,r_2,\theta_1,\theta_2)=|x-y|=\sqrt{r_1^2+r_2^2-2r_1r_2\cos(\theta_1-\theta_2)},
\end{equation}
and
\begin{equation} \label{d-s}
d_s(r_1,r_2,\theta_1,\theta_2)=|{\bf n}_s|=\sqrt{r_1^2+r_2^2+2  r_1r_2\, \cosh s},\quad s\in [0,+\infty).
\end{equation}
Then the kernel of the Bochner-Riesz means operator $S_{\lambda}^\delta(\LL_{\A})$  can be written as 
\begin{equation}\label{ker:BR'}
\begin{split}
S_{\lambda}^\delta(\LL_{\A})(x,y)=&\Big(1-\frac{\LL_{\A}}{\lambda^2}\Big)^\delta_{+}(x,y)\\
&=G_{\lambda}(\delta; r_1,\theta_1;r_2,\theta_2)+ D_{\lambda}(\delta; r_1,\theta_1;r_2,\theta_2).
 \end{split}
\end{equation}
Here
\begin{equation}\label{ker:BR-G}
\begin{split}
&G_{\lambda}(\delta; r_1,\theta_1;r_2,\theta_2)
\\&=\lambda^2\Big[\frac{a_1(\lambda d)e^{i\lambda d}}{(1+\lambda d)^{\frac{3}2+\delta}}+\frac{a_2(\lambda d)e^{-i\lambda d}}{(1+\lambda d)^{\frac{3}2+\delta}}+O((1+\lambda d)^{-3})\Big] \times
A_{\alpha}(\theta_1,\theta_2),
\end{split}
\end{equation}
and
\begin{equation}\label{ker:BR-D}
\begin{split}
&D_{\lambda}(\delta; r_1,\theta_1;r_2,\theta_2)\\
&=\lambda^2\int_0^\infty \Big[\frac{a_1(\lambda d_s)e^{i\lambda d_s}}{(1+\lambda d_s)^{\frac{3}2+\delta}}+\frac{a_2(\lambda d_s)e^{-i\lambda d_s}}{(1+\lambda d_s)^{\frac{3}2+\delta}}+O((1+\lambda d_s)^{-3})\Big]  \, B_{\alpha}(s,\theta_1,\theta_2) ds,
\end{split}
\end{equation}
where  $a_j (j=1,2)$ are bounded from below near infinity and satisfy
\begin{equation}\label{est:aj}
\Big|\big(\frac{\partial}{\partial r}\big)^N a_j(r)\Big|\leq C_N r^{-N}, \quad r>0,\, \forall N\geq 0.
\end{equation} 

\end{proposition}

\begin{remark} The kernel of the Bochner–Riesz means \eqref{ker:BR} consists of two terms, $G_{\lambda}(\delta; r_1,\theta_1;r_2,\theta_2)$ and $D_{\lambda}(\delta; r_1,\theta_1;r_2,\theta_2)$, 
which correspond to the geometric geodesic and the diffractive geodesic \cite{FW}, respectively. It is evident that $A_{\alpha}(\theta_1,\theta_2)$ and $B_{\alpha}(s,\theta_1,\theta_2)$ appear in the kernel, presenting new obstacles to obtaining $L^p$-boundedness.
\end{remark}

\begin{proof}[Proof of Proposition \ref{prop: ker-BR}] Define
\begin{equation}
\begin{split}
G_{\lambda}(\delta; r_1,\theta_1;r_2,\theta_2)
&=\frac1{\pi}\int_0^\infty \rho \Big(1-\frac{\rho^2}{\lambda^2}\Big)^\delta_{+} 
 \int_{\mathbb{S}^1} e^{-i\rho (x-y)\cdot\omega} d\sigma_\omega d\rho\, A_{\alpha}(\theta_1,\theta_2) ,
 \end{split}
\end{equation}
and
\begin{equation}
\begin{split}
D_{\lambda}(\delta; s; r_1,\theta_1;r_2,\theta_2)
&=\frac1{\pi}\int_0^\infty \rho \Big(1-\frac{\rho^2}{\lambda^2}\Big)^\delta_{+} \int_0^\infty \int_{\mathbb{S}^1} e^{-i\rho {\bf n}_s\cdot\omega} d\sigma_\omega
\, B_{\alpha}(s,\theta_1,\theta_2) ds\, d\rho.
 \end{split}
\end{equation}
To prove this proposition, {in view of} \eqref{ker:BR}, it suffices to prove \eqref{ker:BR-G} and \eqref{ker:BR-D}, which follow directly from the stationary phase method. Before doing so, we recall a known result concerning the kernel of the Bochner–Riesz mean for $S_{\lambda}^\delta(-\Delta)$ on $\R^n$
\begin{equation}\label{ker:BR-Rn}
K^\delta_\lambda(x)=\int_{\R^n} e^{ix\cdot \xi} \big(1-|\xi|^2/\lambda^2 \big)_+^\delta\, d\xi=\lambda^n K^\delta_1(\lambda x).
\end{equation}
On the one hand, we have
\begin{equation}
K^\delta_1( x)=\int_{\R^n} e^{ix\cdot \xi} \big(1-|\xi|^2 \big)_+^\delta\, d\xi=\pi^{-\delta}(2\pi)^{\frac n2+\delta}\Gamma(1+\delta)|x|^{-\frac n2-\delta} J_{\frac{n}2+\delta}(|x|),
\end{equation}
for which we refer to \cite[\S 6.19, P430]{Stein}.
On the other hand, one has the complete asymptotic expansion of the Bessel function $J_{\nu}(r)$ (e.g. \cite[(15) Chapter 8, P338]{Stein}. )
$$J_{\nu}(r)\sim r^{-\frac12}e^{ir}\sum_{j=0}^\infty a_j r^{-j}+r^{-\frac12}e^{-ir}\sum_{j=0}^\infty b_j r^{-j}, \quad \text{as}\,\, r\to+\infty$$
for suitable coefficients $a_j$ and $b_j$. Therefore, these imply that $K^\delta_1( x)$ is bounded,  and that the asymptotic expansion of $K^\delta_1( x)$ holds
\begin{equation}
K^\delta_1( x)\sim |x|^{-(n+1)/2-\delta}\Big(e^{i|x|}\sum_{j=0}^\infty \alpha_j |x|^{-j}+e^{-i|x|}\sum_{j=0}^\infty \beta_j |x|^{-j}\Big)
\end{equation}
as $|x|\to+\infty$, for suitable coefficients $\alpha_j$ and $\beta_j$. Based on the above discussion, or by using the argument of \cite[Lemma 2.3.3]{Sogge},
we have
\begin{lemma}\label{lem:keyBR} The kernel of $S_{1}^\delta(-\Delta):=S^\delta(-\Delta)$ can be written as
\begin{equation}
K^\delta(x)=K^\delta_1(x)=\frac{a_1(|x|)e^{i|x|}}{(1+|x|)^{\frac{n+1}2+\delta}}+\frac{a_2(|x|)e^{-i|x|}}{(1+|x|)^{\frac{n+1}2+\delta}}+O((1+|x|)^{-n-1})
\end{equation}
where the $a_j (j=1,2)$ are bounded from below near infinity and satisfy
\begin{equation*}
\Big|\big(\frac{\partial}{\partial r}\big)^N a_j(r)\Big|\leq C_N r^{-N}, \quad \forall N\geq 0.
\end{equation*} 
\end{lemma}

We first consider \eqref{ker:BR-G} by rewriting 
\begin{equation}
\begin{split}
G_{\lambda}(\delta; r_1,\theta_1;r_2,\theta_2)
&=\frac1{\pi}\int_{\R^2} \Big(1-\frac{|\xi|^2}{\lambda^2}\Big)^\delta_{+} 
 e^{-i (x-y)\cdot\xi} d\xi\, A_{\alpha}(\theta_1,\theta_2),
 \end{split}
\end{equation}
which is very close to the Bochner-Riesz means \eqref{ker:BR-Rn} on $\R^2$. Recalling $d=|x-y|$ in \eqref{d-j}, {this lemma together with} \eqref{ker:BR-Rn} {gives} \eqref{ker:BR-G}
\begin{equation}
\begin{split}
&G_{\lambda}(\delta; r_1,\theta_1;r_2,\theta_2)\\
&=\lambda^2\Big[\frac{a_1(\lambda d)e^{i\lambda d}}{(1+\lambda d)^{\frac{3}2+\delta}}+\frac{a_2(\lambda d)e^{-i\lambda d}}{(1+\lambda d)^{\frac{3}2+\delta}}+O((1+\lambda d)^{-3})\Big] \times
A_{\alpha}(\theta_1,\theta_2).
 \end{split}
\end{equation}

Next we prove \eqref{ker:BR-D} by rewriting 
\begin{equation}
\begin{split}
&D_{\lambda}(\delta; r_1,\theta_1;r_2,\theta_2)
=\frac1{\pi}\int_0^\infty \int_{\R^2} \Big(1-\frac{|\xi|^2}{\lambda^2}\Big)^\delta_{+} 
 e^{-i {\bf n}_s\cdot\xi} d\xi 
\, B_{\alpha}(s,\theta_1,\theta_2) ds.
 \end{split}
\end{equation}
Similarly, we use Lemma \ref{lem:keyBR} to obtain 
\begin{equation}
\begin{split}
&D_{\lambda}(\delta; r_1,\theta_1;r_2,\theta_2)\\
&=\lambda^2\int_0^\infty \Big[\frac{a_1(\lambda d_s)e^{i\lambda d_s}}{(1+\lambda d_s)^{\frac{3}2+\delta}}+\frac{a_2(\lambda d_s)e^{-i\lambda d_s}}{(1+\lambda d_s)^{\frac{3}2+\delta}}+O((1+\lambda d_s)^{-3})\Big]  \, B_{\alpha}(s,\theta_1,\theta_2) ds,
 \end{split}
\end{equation}
as desired.

\end{proof}

We conclude this section by recalling two basic lemmas concerning the decay estimates for oscillatory integrals; see Stein \cite{Stein}.

\begin{lemma}[The estimates on oscillatory integrals,\cite{Stein} ] \label{lem:absdec}
Let $\phi$ and $\psi$ be smooth functions so that $\psi$ has compact support in $(a,b)$, and $\phi'(x)\neq 0$ for all $x\in[a,b]$. Then,
\begin{equation}\label{equ:lamdec}
  \big|\int_a^b e^{i\lambda \phi(x)}\psi(x)\;dx\big|\lesssim (1+\lambda)^{-K},
\end{equation}
for all $K\geq0$.
\end{lemma}

We also need the following Van der Corput lemma.
\begin{lemma}[Van der Corput,\cite{Stein} ] \label{lem:VCL} Let $\phi$ be real-valued and smooth in $(a,b)$, and suppose that $|\phi^{(k)}(x)|\geq1$ for all $x\in (a,b)$. Then
\begin{equation}
\left|\int_a^b e^{i\lambda\phi(x)}\psi(x)dx\right|\leq c_k\lambda^{-1/k}\left(|\psi(b)|+\int_a^b|\psi'(x)|dx\right)
\end{equation}
holds when (i) $k\geq2$ or (ii) $k=1$ and $\phi'(x)$ is monotonic. Here $c_k$ is a constant depending only on $k$.
\end{lemma}




\section{The proof of the Main Theorem}\label{sec:thmmain}

In this section, we prove the sufficiency part of \eqref{est:BR} in  Theorem \ref{thm:LA0} by using two key {propositions} \ref{prop:TGj} and \ref{prop:TDj} below,  whose proofs are postponed to the next sections.
By the scaling invariance of the operator $\LL_{{\A}}$, it suffices to prove \eqref{est:BR} for $\lambda=1$; that is,
\begin{equation}\label{est:br1}
\|S_{1}^\delta(\LL_{\A}) f\|_{L^p(\R^2)}\leq C\|f\|_{L^p(\R^2)} 
\end{equation}
whenever $
\delta>\delta_c(p,2)=\max\big\{0, 2\big|1/2-1/p\big|-1/2\big\}$.
Indeed,  from \eqref {ker:BR}, one can see that
$S_{\lambda}^\delta(\LL_{\A})(x,y)=\lambda^2 S_{1}^\delta(\LL_{\A})(\lambda x, \lambda y),$
hence \eqref{est:br1} implies \eqref{est:BR}. We briefly write $S^\delta(\LL_{\A})=S_{1}^\delta(\LL_{\A})$.\vspace{0.1cm}

From Proposition \ref{prop: ker-BR} with $\lambda = 1$, using \eqref{d-j} and \eqref{d-s} and setting $a_+ = a_1$, $a_- = a_2$,
we obtain
\begin{align}\label{eq:log1}
S^\delta(\LL_{\A})(x,y)=&\sum_{\pm}\big(G^{\pm}_1+D^{\pm}_1\big)(r_1,\theta_1;r_2,\theta_2)+\big(G_2+D_2\big)(r_1,\theta_1;r_2,\theta_2)
\end{align}
where the kernels are defined by
\begin{equation}\label{ker-G}
\begin{split}
G^{\pm}_1(r_1,\theta_1;r_2,\theta_2)&=e^{\pm i|x-y|} (1+|x-y|)^{-\frac32-\delta} a_\pm (|x-y|)A_\alpha(\theta_1,\theta_2),\\
G_2(r_1,\theta_1;r_2,\theta_2)&=b(|x-y|)A_\alpha(\theta_1,\theta_2),
\end{split}
\end{equation}
and
\begin{equation}\label{ker-D}
\begin{split}
D^\pm_1(r_1,\theta_1;r_2,\theta_2)&= \int_0^\infty e^{\pm i|{\bf n}_s|} (1+|{\bf n}_s|)^{-\frac32-\delta} a_\pm(|{\bf n}_s|)\, B_\alpha(s,\theta_1,\theta_2)\;ds,\\
D_2(r_1,\theta_1;r_2,\theta_2)&= \int_0^\infty b(|{\bf n}_s|)\, B_\alpha(s,\theta_1,\theta_2)\;ds.
\end{split}
\end{equation}
Here $A_\alpha(\theta_1,\theta_2)$, $B_\alpha(s,\theta_1,\theta_2)$ are as in Proposition \ref{prop:res-ker}, and the functions $a_\pm$ satisfy \eqref{est:aj}, and $b$ satisfies
\begin{equation}\label{est:b}
|b(r)|\leq C(1+r)^{-3}.
\end{equation}
Hence, we write that
\begin{align*}
   S^\delta(\LL_{\A})f(x)
=\big(T_{G^+_1}f+T_{G^-_1}f+T_{G_2}f+T_{D^+_1}f+T_{D^-_1}f+T_{D_2}f\big)(x),
\end{align*}
where
$$T_{K}f(x)=\int_0^\infty \int_0^{2\pi }K(r_1,\theta_1;r_2,\theta_2) f(r_2,\theta_2)  d\theta_2 \,r_2 dr_2.$$
Thus, to prove \eqref{est:br1}, it suffices to prove that there exists a constant $C$ such that
\begin{equation}\label{equ:goalredu}
\|T_{K}\|_{L^p(\R^2)\to L^p(\R^2)}\leq C, \quad K\in\{G^\pm_1, G_2, D^\pm_1, D_2\}.
\end{equation}

\subsection{The estimates of $T_{G_2}$ and $T_{D_2}$.}

We first consider the two simpler cases, $T_{G_2}$ and $T_{D_2}$.  A direct computation yields the following two estimates:
\begin{equation}\label{est:AB}
\begin{split}
|{\bf n}_s|&\geq |x-y|, \\
|A_{\alpha}(\theta_1,\theta_2)|&+\int_0^\infty \big|B_{\alpha}(s,\theta_1,\theta_2)\big| ds\leq C.
\end{split}
\end{equation}
Indeed, to estimate the integral  $\int_0^\infty \big|B_{\alpha}(s,\theta_1,\theta_2)\big| ds$, from \eqref{B-al}, we have used the facts (proved in \cite{FZZ}) that for $\alpha\in(-1,1)\backslash\{0\}$
  \begin{align}\label{equ:ream1}
    \int_0^\infty e^{-|\alpha|s}\;ds\lesssim&1,\\\label{equ:ream2}
    \int_0^\infty  \Big|\frac{(e^{-s}-\cos(\theta_1-\theta_2+\pi))\sinh(\alpha s)}{\cosh(s)-\cos(\theta_1-\theta_2+\pi)}\Big|\;ds\lesssim&1,
    \\\label{equ:ream3}
     \int_0^\infty  \Big|\frac{\sin(\theta_1-\theta_2+\pi)\cosh(\alpha s)}{\cosh(s)-\cos(\theta_1-\theta_2+\pi)}\Big|\;ds\lesssim&1.
  \end{align}

By \eqref{est:b} and  Young’s inequality, we deduce that for $1\leq p\leq \infty$ 
\begin{equation}\label{equ:T2est}
\begin{split}
&\|T_{G_2}f\|_{L^p(\R^2)}+\|T_{D_2}f\|_{L^p(\R^2)}\\
&\leq \Big\|\int_{\R^2} (1+|x-y|)^{-3} |f(y)|dy\Big\|_{L^p(\R^2)}\\
&\qquad+\Big\|\int_{\R^2} \int_0^\infty (1+|{\bf n}_s|)^{-3} |B_{\alpha}(s,\theta_1,\theta_2)\big| ds |f(y)|dy\Big\|_{L^p(\R^2)}\\
&\leq \Big\|\int_{\R^2} (1+|x-y|)^{-3} |f(y)|dy\Big\|_{L^p(\R^2)}\leq C\|f\|_{L^p(\R^2)}.
\end{split}
\end{equation}

\subsection{The estimate of $ T_{G^\pm_1}$.} In this subsection, we prove
\begin{equation}\label{equ:tg1red}
  \|T_{G^\pm_1}\|_{L^{p}(\R^2)\to L^{p}(\R^2)}\leq C.
\end{equation}
To this end, recall $A_\alpha(\theta_1,\theta_2)$ from \eqref{A-al}. By dropping the factors $e^{\pm i\alpha\pi}$, it suffices to prove
\begin{equation}\label{equ:tg2red}
\begin{split}
\Big\|\int_0^\infty \int_0^{2\pi } K_{G_1^\pm}(r_1,r_2;\theta_1-\theta_2) f(r_2,\theta_2) r_2 dr_2 d\theta_2\Big\|_{L^p(\R^2)}\leq C\|f\|_{L^p(\R^2)},
\end{split}
\end{equation}
where
\begin{equation}
\begin{split}
K_{G_1^\pm}(r_1,r_2;\theta_1-\theta_2)&=e^{\pm i|x-y|} (1+|x-y|)^{-\frac32-\delta} a_\pm (|x-y|)
\mathbbm{1}_{I}(|\theta_1-\theta_2|),
\end{split}
\end{equation}
and $I=[0,\pi]$ and $[\pi,2\pi)$. We will only consider the case $I=[0,\pi]$. Indeed, if $I=[\pi, 2\pi)$, by replacing $\theta_2$ with $\theta_2+\pi$, we note that $|x-y|=\sqrt{r_1^2+r_2^2-2r_1r_2\cos(\theta_1-\theta_2)}$, and the same argument applies. From now on, we fix $I=[0,\pi]$.
Using the standard partition of unity
$$\beta_0(r)=1-\sum_{j\geq1}\beta(2^{-j}r),\quad \beta\in\mathcal{C}_c^\infty\big(\big[\tfrac38,\tfrac43\big]\big), \quad \beta_0\in\mathcal{C}_c^\infty\big([0,1]\big),$$ we decompose
\begin{equation}\label{KGj}
\begin{split}
K_{G_1^\pm}(r_1,r_2;\theta_1-\theta_2)
=\sum_{j\geq 0}K_{G_1^\pm}^j(r_1,r_2;\theta_1-\theta_2),
\end{split}
\end{equation}
where
$$K_{G_1^\pm}^j(r_1,r_2;\theta_1-\theta_2):=\beta(2^{-j}|x-y|)K_{G_1^\pm}(r_1,r_2;\theta_1-\theta_2), \quad j\geq1$$
and $$K_{G_1^\pm}^0(r_1,r_2;\theta_1-\theta_2):=\beta_0(|x-y|)K_{G_1^\pm}(r_1,r_2;\theta_1-\theta_2).$$
Then we have
\begin{proposition}\label{prop:TGj} Let $T_{G_1^\pm}^j$ be the operator associated with the kernel $K_{G_1^\pm}^j(r_1,r_2;\theta_1-\theta_2)$
given in \eqref{KGj} and let $p>4$.
Then, {for any  $0<\epsilon'\ll1$}
 \begin{equation}\label{est:TGj}
 \|T_{G_1^\pm}^j\|_{L^p\to L^p}\leq C {2^{j[\delta_c(p,2)-\delta+\epsilon']}}, \quad  \|T_{G_1^\pm}^j\|_{L^2\to L^2}\leq C 2^{-j\delta} \quad \forall j\geq0.
\end{equation}
\end{proposition}

We postpone the proof of \eqref{est:TGj} to {the} next section.
Now, by using Proposition \ref{prop:TGj} and { by choosing a sufficiently small $\epsilon'$ such that $0<\epsilon'\leq -\frac12[\delta_c(p,2)-\delta]$}, and then summing a geometric series, for $p>4$ and $p=2$, we obtain \eqref{equ:tg2red}, hence
\eqref{equ:tg1red} follows. For $2\leq p\leq 4$,  we obtain \eqref{equ:tg1red} by interpolation. The remaining cases $1\leq p< 2$ in \eqref{equ:tg1red}  follow by duality.

\subsection{The estimate of $ T_{D^\pm_1}$.} In this subsection, we prove
\begin{equation}\label{equ:td1red}
  \|T_{D^\pm_1}\|_{L^{p}(\R^2)\to L^{p}(\R^2)}\leq C.
\end{equation}
We only consider the $+$ case in the following, since the same argument works for the $-$ case.
To lighten the notation, we replace ${D_1^\pm}$ by $D$ and consider only the $+$ case. More precisely, it suffices to prove
\begin{equation}\label{equ:td2red}
\begin{split}
\Big\|\int_0^\infty \int_0^{2\pi } K_{D}(r_1,r_2;\theta_1-\theta_2) f(r_2,\theta_2) r_2 dr_2 d\theta_2\Big\|_{L^p(\R^2)}\leq C\|f\|_{L^p(\R^2)},
\end{split}
\end{equation}
where 
\begin{equation}
\begin{split}
K_{D}(r_1,\theta_1;r_2,\theta_2)&= \int_0^\infty e^{ i|{\bf n}_s|} (1+|{\bf n}_s|)^{-\frac32-\delta} a(|{\bf n}_s|)\, B_\alpha(s,\theta_1,\theta_2)\;ds.
\end{split}
\end{equation}
and $a$ satisfies \eqref{est:aj}  and
\begin{equation*}
\begin{split}
&B_{\alpha}(s,\theta_1,\theta_2)= -\frac{1}{4\pi^2}e^{-i\alpha(\theta_1-\theta_2)+i\int_{\theta_2}^{\theta_{1}} A(\theta') d\theta'}  \Big(\sin(|\alpha|\pi)e^{-|\alpha|s}\\
&\qquad +\sin(\alpha\pi)\frac{(e^{-s}-\cos(\theta_1-\theta_2+\pi))\sinh(\alpha s)-i\sin(\theta_1-\theta_2+\pi)\cosh(\alpha s)}{\cosh(s)-\cos(\theta_1-\theta_2+\pi)}\Big).
 \end{split}
\end{equation*}
Note that $K_{D}(r_1,\theta_1;r_2,\theta_2)=0$ when $\alpha=0$. We always assume $\alpha\neq 0$.\vspace{0.2cm}

It suffices to prove the three estimates, for $\ell=1,2,3$ and $1\leq p\leq \infty$
\begin{equation}\label{equ:kdellred}
\Big\|\int_0^\infty \int_0^{2\pi } K^{\ell}_D(r_1,r_2;\theta_1-\theta_2) f(r_2,\theta_2) r_2 dr_2 d\theta_2\Big\|_{L^p(\R^2)}\leq C\|f\|_{L^p(\R^2)},
\end{equation}
where
\begin{equation}\label{kerD12}
\begin{split}
K^1_D(r_1,r_2;\theta_1-\theta_2)&=\int_0^\infty e^{ i|{\bf n}_s|} (1+|{\bf n}_s|)^{-\frac32-\delta}  a(|{\bf n}_s|)\, e^{-|\alpha|s}\;ds,\\
K^2_D(r_1,r_2;\theta_1-\theta_2)&=\int_0^\infty e^{ i|{\bf n}_s|} (1+|{\bf n}_s|)^{-\frac32-\delta}  a(|{\bf n}_s|)\, \\
 & \qquad\qquad\qquad\times \frac{(e^{-s}-\cos(\theta_1-\theta_2+\pi))\sinh(\alpha s)}{\cosh(s)-\cos(\theta_1-\theta_2+\pi)}\;ds,
 \end{split}
\end{equation}
and
\begin{equation}\label{kerD3}
\begin{split}
K^3_D(r_1,r_2;\theta_1-\theta_2)&=\int_0^\infty e^{ i|{\bf n}_s|} (1+|{\bf n}_s|)^{-\frac32-\delta}  a(|{\bf n}_s|)\, \\
&\qquad\qquad\times\frac{\sin(\theta_1-\theta_2+\pi)\cosh(\alpha s)}{\cosh(s)-\cos(\theta_1-\theta_2+\pi)}\;ds.
\end{split}
\end{equation}
Using the partition of unity again
$$\beta_0(r)=1-\sum_{j\geq1}\beta(2^{-j}r),\quad \beta\in\mathcal{C}_c^\infty\big(\big[\tfrac38,\tfrac43\big]\big),$$ we decompose
\begin{equation}\label{KDlj}
K^{\ell}_D(r_1,r_2;\theta_1-\theta_2)
=:\sum_{j\geq 0}K^{\ell,j}_D(r_1,r_2;\theta_1-\theta_2),
\end{equation}
where for $j\geq1$ and $\ell=1,2,3$
$$K_D^{\ell,j}(r_1,r_2;\theta_1-\theta_2)=\beta(2^{-j}(r_1+r_2))K^\ell_D(r_1,r_2;\theta_1-\theta_2)$$
and $K_D^{\ell,0}(r_1,r_2;\theta_1-\theta_2)=\beta_0(r_1+r_2)K^\ell_D(r_1,r_2;\theta_1-\theta_2)$.\vspace{0.2cm}

For our {purposes}, we need the following proposition, which is an analogue of Proposition \ref{prop:TGj}.
\begin{proposition}\label{prop:TDj} For $\ell=1,2,3$, let $T_{D}^{\ell,j}$ be the operator associated with kernel $K_{D}^{\ell,j}(r_1,r_2;\theta_1-\theta_2)$ given in \eqref{KDlj} and let $p>4$.
Then, for $j\geq0$ and $0<\epsilon'\ll1$, it holds
 \begin{equation}\label{est:TDj}
 \begin{split}
    \|T_{D}^{\ell,j}\|_{L^2\to L^2}&\leq C 2^{-j\delta} , \\
     \|T_{D}^{\ell,j}\|_{L^p\to L^p}&\leq C 2^{j[\delta_c(p,2)-\delta]}\times\begin{cases} 1\quad &6\leq p\\
    2^{j\epsilon'}, \quad &4< p<6.
    \end{cases}
    \end{split}
\end{equation}
\end{proposition}
With  Proposition \ref{prop:TDj} in hand, we obtain \eqref{equ:kdellred} by using the same argument as in the proof of the $L^p$-bounds of $T_{G_1^\pm}^j$ from Proposition \ref{prop:TGj};
see the details after Proposition \ref{prop:TGj}. We remark here that even though there is an {arbitrarily} small loss when $4<p<6$, one can sum the geometric series
by choosing a sufficiently small $\epsilon'$ such that $0<\epsilon'\leq -\frac12[\delta_c(p,2)-\delta]$. 

\vspace{0.2cm}

Therefore, we complete the proof of the sufficiency part of \eqref{est:BR} in Theorem \ref{thm:LA0},  provided that   Proposition \ref{prop:TGj} and Proposition \ref{prop:TDj} {hold}.

\medskip

\section{The proof of   Proposition \ref{prop:TGj}}\label{sec:TGj}

The main task of this section is to prove Proposition \ref{prop:TGj} by applying the theory of oscillatory integrals due to Stein \cite{Stein} and H\"ormander \cite{Hor}; we also refer to Sogge \cite[Corollary 2.2.3]{Sogge1}. However, the kernel contains a jump function $\mathbbm{1}_{[0,\pi]}(|\theta_1-\theta_2|)$, which is not smooth. Consequently, the estimates do not follow directly from the standard stationary phase theory, and special care is required. To lighten the notation, we replace ${G_1^\pm}$ by $G$ and consider only the $+$ case. \vspace{0.2cm}

We first consider the simplest case when $j=0$. Since $\beta_0\in L^{1}(\R^2)$, for $1\leq p\leq \infty$, Young's inequality gives
$$ \|T_{G}^0f\|_{L^p}\leq \Big\|\int_{\R^2} \beta_0(|x-y|) f(y) dy\Big\|_{L^p(\R^2)}\leq C\|f\|_{L^p(\R^2)}.
$$

From now on, we assume that $j\geq 1$.
Let $0<\epsilon\ll 1$ be sufficiently small. Since $[0,2\pi]$ is compact, we cover
$[0,2\pi]$ by $\Theta_\epsilon=\{\theta: |\theta-\theta_{k,0}|\leq \epsilon\}$ with fixed $\{\theta_{k,0}\}_{k=1}^{N}\subset [0,2\pi]$ where $N$ is finite.
 Noting that
$$K_G^j(r_1,r_2;\theta_1-\theta_2)=\beta(2^{-j}|x-y|)K_G(r_1,r_2;\theta_1-\theta_2),$$
and using translation invariance in the $L^q_\theta$ norm,  it suffices to prove that  \eqref{est:TGj} holds under the assumption
\begin{equation}\label{assu:f}
\text{supp} f\subset \{(r_2,\theta_2): r_2\in(0,\infty),  \theta_2\in[0,\epsilon)\},\quad 0<\epsilon\ll 1.
\end{equation}
Note that if $\theta_2\in[0,\epsilon]$, then
\begin{equation}\label{I-theta}
\mathbbm{1}_{[0,\pi]}(|\theta_1-\theta_2|)=
\begin{cases}1\quad \text{when}\,\theta_1\in [0,\pi],\\
0\quad \text{when}\,\theta_1\in [\pi+\epsilon,2\pi].
\end{cases}
\end{equation}
We remark that there is a jump when $\theta_1\in [\pi,\pi+\epsilon]$, which is the most delicate case.
For our purposes, we divide the proof into three cases.

{\bf Case 1: $\theta_1\in[\pi+\epsilon,2\pi].$}  Since
$$\mathbbm{1}_{[\pi+\epsilon,2\pi]}(\theta_1) K_G^j(r_1,r_2;\theta_1-\theta_2)=0,$$
we have $\mathbbm{1}_{[\pi+\epsilon,2\pi]}(\theta_1)T_{G}^j  f=0$. Obviously, for all $1\leq p\leq +\infty$
$$\Big\|\mathbbm{1}_{[\pi+\epsilon,2\pi]}(\theta_1)T_{G}^j f\Big\|_{L^p(\R^2)}\leq C 2^{j[\delta_c(p,2)-\delta]}\|f\|_{L^p}.$$ \vspace{0.2cm}

{\bf Case 2: $\theta_1\in[0,\pi].$}
In this case, by \eqref{I-theta}, we can drop the jump function $\mathbbm{1}_{[0,\pi]}(|\theta_1-\theta_2|)$ to obtain 
 $$\mathbbm{1}_{[0,\pi]}(\theta_1) K_G^j(r_1,r_2;\theta_1-\theta_2)=\mathbbm{1}_{[0,\pi]}(\theta_1) \beta(2^{-j}|x-y|)e^{i|x-y|} (1+|x-y|)^{-\frac32-\delta} a(|x-y|).$$
Thus, the jump function disappears. Hence, we can apply the theory of {oscillatory} integrals (e.g. \cite[Lemma 2.3.4]{Sogge1}) to estimate
\begin{equation}
\begin{split}
&\Big\|\mathbbm{1}_{[0,\pi]}(\theta_1)T_{G}^j f\Big\|_{L^p(\R^2)}\\
\leq& \Big\|\int_{\R^2}\beta(2^{-j}|x-y|)e^{i|x-y|} (1+|x-y|)^{-3/2-\delta} a(|x-y|)f(y) dy \Big\|_{L^p(\R^2)}.
\end{split}
\end{equation}
Define 
$$K^j_G(x)=\beta(2^{-j}|x|)e^{i|x|} (1+|x|)^{-3/2-\delta} a(|x|).$$
We observe that the $L^p-L^p$ operator norm of convolution with $K^j_G$ equals the norm of convolution with the dilated kernel
\begin{equation}
\begin{split}
\tilde{K}^j_G(x)=2^{2j}K^j_G(2^j x)&=2^{j(\frac12-\delta)} \beta(|x|)e^{i2^j |x|} (2^{-j}+|x|)^{-3/2-\delta} a(2^j |x|)\\
&=:2^{j(\frac12-\delta)}b(2^j, x)e^{i2^j |x|}.
\end{split}
\end{equation}
Recalling that $a$ satisfies \eqref{est:aj}, we see  $\big|(\frac{\partial}{\partial x})^\alpha b(2^j, x)\big|\leq C_\alpha$. Then 
\cite[Lemma 2.3.4 or $(2.3.9')$]{Sogge1} implies
$$\|\tilde{K}^j_G(x)\ast f\|_{L^p(\R^2)}\leq C 2^{j(\frac12-\delta-\frac2p)}\|f\|_{L^p(\R^2)}=C2^{j[\delta_c(p,2)-\delta]}\|f\|_{L^p}, \quad p>4.$$
For $p=2$, we use \cite[Theorem 2.1.1]{Sogge1} to obtain 
$$\|\tilde{K}^j_G(x)\ast f\|_{L^2(\R^2)}\leq C 2^{j(\frac12-\delta)} 2^{-\frac j2} \|f\|_{L^2(\R^2)}=C 2^{-j\delta} \|f\|_{L^2(\R^2)}.$$
Hence, we have proved \eqref{est:TGj} in this case.\vspace{0.2cm}

{\bf Case 3: $\theta_1\in[\pi,\pi+\epsilon].$} For convenience, we again define
$$b(2^j, |x-y|)=\beta(|x-y|)(2^{-j}+ |x-y|)^{-3/2-\delta} a(2^j|x-y|), $$
thus, for every $\alpha\geq 0$, we have
\begin{equation}\label{est:bj}
\big|\partial_r^\alpha b(2^j, r)\big|\leq C_\alpha
\end{equation} due to  the compact support of $\beta$ and the fact that $a$ satisfies \eqref{est:aj}.
To prove \eqref{est:TGj} in this case, we first
recall the kernel of $T_{G}^j$
$$K^j_G(r_1,r_2;\theta_1-\theta_2)=\beta(2^{-j}|x-y|)e^{i|x-y|} (1+ |x-y|)^{-3/2-\delta} a(|x-y|)
\mathbbm{1}_{[0,\pi]}(|\theta_1-\theta_2|),$$
and we rewrite
\begin{equation}\label{tildeK}
\begin{split}
&\tilde{K}^j_G(r_1,r_2;\theta_1-\theta_2)=K^j_G(2^jr_1,2^jr_2;\theta_1-\theta_2)
\\&=\beta(|x-y|)e^{i2^j |x-y|} (2^{-j}+ |x-y|)^{-3/2-\delta} a(2^j|x-y|)\mathbbm{1}_{[0,\pi]}(|\theta_1-\theta_2|)\\
&=2^{-j(\frac32+\delta)} b(2^j, |x-y|)e^{i2^j |x-y|}\mathbbm{1}_{[0,\pi]}(|\theta_1-\theta_2|).
\end{split}
\end{equation}
We use the dilated kernel $\tilde{K}^j_G$ to define the operator
\begin{equation}\label{tildeT}
\begin{split}
\tilde{T}_{G}^j f &=\int_0^\infty\int_0^{2\pi} \tilde{K}^j_G(r_1,r_2;\theta_1-\theta_2)\eta(\theta_1-\theta_2) f(r_2,\theta_2) \,d\theta_2 \,r_2 dr_2,
\end{split}
\end{equation}
where $ \eta(s) \in C_c^\infty ([\pi-2\epsilon,\pi+2\epsilon])$ satisfies $\eta(s)=1$ for $s\in[\pi-\epsilon,\pi+\epsilon]$.
Then one can verify that
$$\mathbbm{1}_{[\pi,\pi+\epsilon]}T_{G}^j \big(\mathbbm{1}_{[0,\epsilon]} f(r_2,\theta_2)\big)(r_1,\theta_1)=2^{2j}\mathbbm{1}_{[\pi,\pi+\epsilon]}\tilde{T}_{G}^j \big(\mathbbm{1}_{[0,\epsilon]} f(2^jr_2,\theta_2)\big)(2^{-j}r_1,\theta_1).$$ To prove \eqref{est:TGj} in this case, 
 it suffices to prove the following proposition.
 \begin{proposition}\label{prop:key} Let $\tilde{T}_{G}^j $ be the operator defined in \eqref{tildeT} with the kernel \eqref{tildeK}. Then
 \begin{equation}\label{est:TGj'2}
\big\|\tilde{T}_{G}^j f\big\|_{L^2(\R^2)} \leq C2^{-j(\frac32+\delta)} 2^{-\frac{j}{2}}\|f\|_{L^{2}(\R^2)},
\end{equation}
and {for $0<\epsilon'\ll 1$}
\begin{equation}\label{est:TGj'4}
\big\|\tilde{T}_{G}^j f\big\|_{L^p(\R^2)} \leq C2^{-j(\frac32+\delta)} {2^{(-\frac{2}{p}+\epsilon')j}}\|f\|_{L^{p}(\R^2)},\quad \text{for}\quad p>4,
\end{equation}
where $f$ satisfies \eqref{assu:f}.
\end{proposition}
\begin{remark}  {Compared the classical result without potential, there is an arbitrarily  small loss in \eqref{est:TGj'4}. However, this is sufficient for our main theorem \ref{thm:LA0}. Indeed, one can improve the arbitrarily  small loss 
to $(1+j)$, see \eqref{eq:sum-low} below. }
\end{remark}

\begin{proof}
We first prove \eqref{est:TGj'2}. We write the operator
\begin{equation*}
\begin{split}
\tilde{T}_{G}^j f(r_1,\theta_1) &=\int_0^\infty\int_0^{2\pi} \tilde{K}^j_G(r_1,r_2;\theta_1-\theta_2)\eta(\theta_1-\theta_2) f(r_2,\theta_2) \,d\theta_2 \,r_2 dr_2\\
&\triangleq\int_0^\infty[\tilde{T}_{G,r_1,r_2}^j f(r_2, \cdot)]( \theta_1) \,r_2 dr_2
\end{split}
\end{equation*}
where the operator is given by
\begin{equation}\label{Tr1r2}
[\tilde{T}_{G,r_1,r_2}^j g(\cdot)]( \theta_1)=\int_0^{2\pi} \tilde{K}^j_G(r_1,r_2;\theta_1-\theta_2) \eta(\theta_1-\theta_2) g(\theta_2) d\theta_2.
\end{equation}
We aim to estimate
\begin{equation}\label{q-q0}
\begin{split}
&\big\|\tilde{T}_{G}^j f\big\|_{L^p(\R^2)} \leq  \int_0^\infty \Big\| [\tilde{T}_{G,r_1,r_2}^j f(r_2, \cdot)]( \theta_1) \Big\|_{L^p(\R^2)} \,r_2 dr_2 .
\end{split}
\end{equation}

Thanks to the assumption on the support of $\tilde{K}^j_G$ in $\theta_1-\theta_2$ and $\theta_2\in[0, \epsilon)$,
we may treat the $\theta_j$ ($j=1,2$) as variables in $\R$. Thus, we can regard $\tilde{T}_{G_1,r_1,r_2}^j$ as a convolution kernel in $\theta$ and use Plancherel's theorem to obtain 
\begin{equation}\label{2-2}
\begin{split}
& \Big\| [\tilde{T}_{G,r_1,r_2}^j f(r_2, \cdot)]( \theta_1) \Big\|_{L^2_{\theta_1}([0,2\pi))}= \Big\| [\tilde{T}_{G,r_1,r_2}^j f(r_2, \cdot)]( \theta_1) \Big\|_{L^2_{\theta_1}(\R)}\\
&= \Big\| \widehat{\tilde{K}^j_G \eta}(\zeta) \hat{f}(\zeta) \Big\|_{L^2 (d\zeta)}\leq C 2^{-j(\frac32+\delta)} (2^jr_1r_2)^{-1/2} \|f\|_{L^2_{\theta}(\R)},
\end{split}
\end{equation}
if we could prove that the corresponding Fourier multiplier satisfies
\begin{align}\nonumber
  &\Big|\int_{\pi-2\epsilon}^{\pi} e^{i\zeta\theta+i2^jd_G(r_1,r_2; \theta)}
 b(2^j, d_G(r_1,r_2; \theta))\eta(\theta)
\;d\theta\Big|\\\label{equ:kernconfrou}
\lesssim& 
(2^jr_1r_2)^{-1/2}.
\end{align}
We remark here that we drop the cutoff function $\mathbbm{1}_{[0,\pi]}(|\theta|) $ by using $\int_{\pi-2\epsilon}^{\pi}$.
To show \eqref{equ:kernconfrou}, we recall
 $$d_G(r_1,r_2; \theta)=|x-y|=\sqrt{r_1^2+r_2^2-2r_1r_2\cos(\theta)}.$$ 
Due to the compact support of $\eta$,  we note that
$\theta\to \pi$ as $\epsilon\to 0$. Hence the fact $|x-y|\sim 1$ implies $r_1+r_2\sim 1$.
Then,  as $\epsilon\to 0$,  we compute that
\begin{equation}\label{1-d}
\begin{split}
\partial_{\theta_1}d_G(r_1,r_2; \theta)&=\frac{r_1r_2\sin(\theta)}{\sqrt{r_1^2+r_2^2-2r_1r_2\cos(\theta)}}\\
&=\frac{r_1r_2}{r_1+r_2}\sin(\theta)+O((r_1r_2)^2(\theta-\pi)^3),
\end{split}
\end{equation}
and
\begin{equation}\label{2-d}
\begin{split}
&\partial^2_{\theta_1}d_G(r_1,r_2; \theta)=\frac{r_1r_2}{r_1+r_2}\cos(\theta)+O((r_1r_2)^2(\theta-\pi)^2).\end{split}
\end{equation}
Thus, from \eqref{2-d}, we see
$$|\partial_{\theta_1}^2d_G(r_1,r_2; \theta_1-\theta_2)|\geq c r_1r_2.$$
Hence, we use the Van der Corput lemma \ref{lem:VCL} to obtain \eqref{equ:kernconfrou}.
Thus, we have proved \eqref{2-2}. By using \eqref{q-q0}, \eqref{2-2} and the facts that $r_1+r_2\sim 1$ and $\theta_2\in[0, \epsilon)$, we show that
\begin{equation*}
\begin{split}
\big\|\tilde{T}_{G}^j f\big\|_{L^2(\R^2)} &\leq  C 2^{-j(\frac32+\delta)} 2^{-\frac j2} \int_0^1 r_2^{-\frac12}\|f(r_2,\theta_2)\|_{L^2_{\theta_2}(\R)} \,r_2 dr_2 \\
& \leq  C 2^{-j(\frac32+\delta)} 2^{-\frac j2} \|f\|_{L^{2}(\R^2)},
\end{split}
\end{equation*}
which is \eqref{est:TGj'2}.\vspace{0.2cm}

We next prove \eqref{est:TGj'4} . We recall \eqref{tildeT} and \eqref{tildeK}
\begin{equation}
\begin{split}
\tilde{T}_{G}^j f &=\int_0^\infty\int_0^{2\pi} \tilde{K}^j_G(r_1,r_2;\theta_1-\theta_2)\eta(\theta_1-\theta_2) f(r_2,\theta_2) \,d\theta_2 \,r_2 dr_2,
\end{split}
\end{equation}
and\footnote{Here we replace $\mathbbm{1}_{[0,\pi]}(|\theta_1-\theta_2|)$ by $\mathbbm{1}_{[0,\pi]}(\theta_1-\theta_2)$ since we are in the case that $\theta_1\in[\pi,\pi+\epsilon]$ and $\theta_2\in[0,\epsilon]$.}
\begin{equation}
\begin{split}
\tilde{K}^j_G(r_1,r_2;\theta_1-\theta_2)=2^{-j(\frac32+\delta)} b(2^j, |x-y|)e^{i2^j |x-y|}\mathbbm{1}_{[0,\pi]}(\theta_1-\theta_2),
\end{split}
\end{equation}
$$b(2^j, |x-y|)=\beta(|x-y|)(2^{-j}+ |x-y|)^{-3/2-\delta} a(2^j|x-y|), $$
where $b(2^j, r)$ satisfies \eqref{est:bj} and $a$ satisfies \eqref{est:aj}.
\begin{lemma}\label{lem:sogge-global}
Let $\lambda=2^{j}$ with $j\geq 1$ and $j\in\Z$. Define
\begin{equation}\label{eq:R}
(R_\lambda g)(x)=\int_{\R^2} e^{i\lambda|x-y|}b(\lambda, |x-y|)g(y)\,dy,
\end{equation}
Then, for $p>4$
\begin{equation}\label{eq:S}
\|R_\lambda g\|_{L^{p}(\R^2)}
\leq C_p\lambda^{-2/p}\|g\|_{L^{p}(\R^2)},
\end{equation}
with $C_p$ independent of $\lambda$.
\end{lemma}

\begin{proof} This is essentially the theory of oscillatory integrals due to Stein \cite{Stein} and H\"ormander \cite{Hor}. 
We adopt Sogge \cite[Lemma 2.3.4 or $(2.3.9')$]{Sogge1} with $\psi_1(x-y)=|x-y|$.
Let $\omega=(x-y)/|x-y|$. On the fixed annulus $|x-y|\sim 1$,
\[
\nabla_x\psi=\omega,
\qquad
\nabla_y\nabla_x\psi=-\frac{I-\omega\otimes\omega}{|x-y|}.
\]
The mixed matrix has rank one, which is the required rank $n-1$ in dimension
two. For fixed $x$, the image of $y\mapsto\nabla_x\psi(x,y)$ is an arc of $S^{1}$
with curvature one; therefore the key geometric hypotheses of Sogge \cite[Lemma 2.3.4]{Sogge1} are satisfied in the present setting.

The amplitude $b(\lambda, |x-y|)$ is smooth and supported in {$\{(x,y)\in\R^2\times\R^2: \frac38 \leq |x-y|\leq\frac43\}$} away from the diagonal $x=y$.
Therefore, the estimate \cite[$(2.3.9')$]{Sogge1} shows \eqref{eq:S}.
\end{proof}

To prove \eqref{est:TGj'4},  if one uses a direct application of the above estimate,  we need to consider the smoothness of two additional factors\footnote{we use $\theta_x$ and $\theta_y$ to replace $\theta_1$ and $\theta_2$, respectively. }  $\eta(\theta_x-\theta_y)$ and $\mathbbm{1}_{[0,\pi]}(\theta_x-\theta_y)$  in the amplitude.  The smoothness of the amplitude $b(\lambda, |x-y|)\eta(\theta_x-\theta_y)$ is \emph{not} justified at the polar origin
$x=0$ or $y=0$, because the polar angle is singular there and the source support
permits one radius to approach zero.  We remove the obstruction by angular Fourier conjugation.

We first treat the term $\eta(\theta_x-\theta_y)$.  Define the periodization
\begin{equation}\label{peri}
\eta_{\mathrm{per}}(t)=\sum_{n\in\Z}\eta(t+2\pi n).
\end{equation} 
Note that $\theta_y\in[0,\epsilon]$ and $\theta_x\in[0,2\pi)$, then the difference
$t=\theta_x-\theta_y$ lies in $[-\epsilon,2\pi)$. Because
$\text{supp}\eta\subset[\pi-2\epsilon,\pi+2\epsilon]\Subset(0,2\pi)$
and $0<\epsilon\ll 1$, no translate of this support other than the zeroth one
meets $[-\epsilon,2\pi)$.  So the periodization \eqref{peri} does not change the operator.
Hence we can identify $\eta_{\mathrm{per}}(t)=\eta(t)$ in this operator.

\begin{lemma}\label{lem:low}
For every $k\in\Z$, define
\[
A_{\lambda,k}f(x)=\int_{\R^2} e^{i\lambda|x-y|}b(\lambda, |x-y|)
\eta_{\mathrm{per}}(\theta_x-\theta_y)
e^{ik(\theta_x-\theta_y)}f(y)\,dy
\]
Then $A_{\lambda,k}$ satisfies
\begin{equation}\label{eq:L}
\|A_{\lambda,k}f\|_{L^p(\R^2)}
\leq C_{p,\eta}\lambda^{-2/p}\|f\|_{L^p(\R^2)},
\end{equation}
uniformly in $k$ and $\lambda$. 
\end{lemma}

\begin{proof}
Since the function $\eta_{\mathrm{per}}$ is smooth on $\mathbb{T}=[0,2\pi]$, we write
\[
\eta_{\mathrm{per}}(t)=\sum_{\ell\in\Z}c_\ell e^{i\ell t},
\qquad \sum_{\ell\in\Z}|c_\ell|<\infty.
\]
Then, in operator norm on $L^p(\R^2)$,
\begin{equation}
\begin{split}
A_{\lambda,k}=\sum_{\ell\in\Z}c_\ell M_{k+\ell}R_\lambda M_{-(k+\ell)},
\end{split}
\end{equation}
where $M_m g(r,\theta)=e^{im\theta}g(r,\theta)$.
Since every $M_m$ is an $L^p$ isometry, the estimate \eqref{eq:S} shows
\begin{equation}
\begin{split}
\|A_{\lambda,k}f\|_{L^p(\R^2)}\leq \sum_{\ell\in\Z}|c_\ell| \| M_{k+\ell}R_\lambda M_{-(k+\ell)} f\|_{L^p(\R^2)}\leq C \lambda^{-2/p}\|f\|_{L^p} \sum_{\ell\in\Z}|c_\ell|.
\end{split}
\end{equation}
The absolute summability proves \eqref{eq:L},
uniformly in $k$.
\end{proof}

Next we treat the term $\mathbbm{1}_{[0,\pi]}(\theta_x-\theta_y)$.  With the Fourier convention
$$\widehat H(k)=(2\pi)^{-1}\int_0^{2\pi}H(t)e^{-ikt}\,dt,$$ we write
\[
H(t)=\mathbbm{1}_{[0,\pi]}(t)
=\frac12+\sum_{\substack{k\in\Z\\ k\text{ odd}}}\frac{e^{ikt}}{\pi i k}
\]
in $L^{q}(\mathbb{T})$, $1<q<\infty$, with midpoint values at the jumps. Rigorously,
use the Abel means
\[
H_\rho(t)=\frac12+\sum_{\substack{k\in\Z\\ k\text{ odd}}}
\rho^{|k|}\frac{e^{ikt}}{\pi i k},\qquad 0<\rho<1.
\]
The Poisson kernels form an approximate identity, so $H_\rho\to H$ in
$L^{p}(\mathbb{T})$ as $\rho\to 1$. We next prove the operator-norm estimate needed to pass from this
scalar convergence to the original two-variable operator.

For $h\in L^{p}(\mathbb{T})$, we define
\begin{equation}\label{eq:Uh}
\begin{aligned}
(\mathcal U_h f)(r_1,\theta_1)
=\int_0^\infty\int_{\mathbb{T}}
&b(\lambda,\dG(r_1,r_2;\theta_1-\theta_2))
e^{i\lambda \dG(r_1,r_2;\theta_1-\theta_2)}\\
&\times\eta_{\mathrm{per}}(\theta_1-\theta_2)
h(\theta_1-\theta_2)f(r_2,\theta_2)\,d\theta_2\,r_2dr_2.
\end{aligned}
\end{equation}

\begin{lemma}\label{lem:Uh}
There is a constant independent of $h,\lambda$ such that
\begin{equation}\label{eq:3}
\norm{\mathcal U_h f}_{L^{p}(\R^2)}
\leq C\norm{h}_{L^{p}(\T)}\norm{f}_{L^{p}(\R^2)}.
\end{equation}
\end{lemma}

\begin{proof}
Fix $r_1,r_2$ and denote the angular convolution kernel by
\[
q_{r_1,r_2}(t)=b(\lambda,\dG(r_1,r_2;t))
e^{i\lambda \dG(r_1,r_2;t)}\eta_{\mathrm{per}}(t)h(t).
\]
By \eqref{est:bj} , one has the uniform bound $|b(\lambda, d_G)|\leq C$. By H\"older's inequality on the
support of $\eta_{\mathrm{per}}$, we obtain
\begin{equation}\label{eq:4}
\norm{q_{r_1,r_2}}_{L^{1}(\T)}
\leq C\norm{\eta_{\mathrm{per}}(t)h(t)}_{L^1(\T)}
\leq C|\supp\eta_{\mathrm{per}}|^{1/p'}\norm{h}_{L^p(\T)}
\leq C\norm{h}_{L^p(\T)}.
\end{equation}
Let $F(r_2)=\norm{f(r_2,\cdot)}_{L^{p}(\T)}$. As argued above \eqref{1-d}, we have $r_1+r_2\sim 1$. By Minkowski's inequality and Young's inequality for angular convolution, we obtain
\begin{equation}\label{eq:5}
\begin{split}
\norm{\mathcal U_h f(r_1,\cdot)}_{L^{p}_{\theta_1}}
&\leq C\norm{h}_{L^p}\int_0^1 F(r_2)r_2\,dr_2\\&\leq C\norm{h}_{L^p}\|F(r_2)\|_{L^p_{r_2 dr_2}}=C\norm{h}_{L^p(\T)}\|f\|_{L^p(\R^2)}.
\end{split}
\end{equation}
Finally take the $L^{p}(r_1dr_1)$ norm of \eqref{eq:5}, again over $[0,1]$. This proves \eqref{eq:3}.
\end{proof}
Applying \eqref{eq:3} to $h=H_\rho-H$, we prove
\begin{equation}\label{eq:7}
\norm{\mathcal U_{H_\rho}-\mathcal U_H}_{L^p(\R^2)\to L^p(\R^2)}
\leq C\norm{H_\rho-H}_{L^p(\T)}\longrightarrow0,\quad \text{as}\quad \rho\to1.
\end{equation}
Thus the Abel convergence is convergence in operator norm for the operators in this
proof, not merely scalar $L^{p}(\T)$ convergence.

Let $p'=p/(p-1)$. Since $(1-p')/p'=-1/p$, the inverse Hausdorff--Young inequality gives
\begin{equation}\label{eq:1}
\Bigl\|\sum_{\substack{|k|>K\\k\text{ odd}}}
\rho^{|k|}\frac{e^{ikt}}{\pi ik}\Bigr\|_{L^{p}(\T)}
\leq C_p\Bigl(\sum_{|k|>K}|k|^{-p'}\Bigr)^{1/p'}
\leq C_p K^{-1/p},
\end{equation}
Let
\[
h_{K,\rho}(t)=\sum_{\substack{|k|>K\\k\text{ odd}}}
\rho^{|k|}\frac{e^{ikt}}{\pi ik}.
\]
Then the estimates \eqref{eq:1} and \eqref{eq:3} imply
\begin{equation}\label{eq:8}
\norm{2^{-j(\frac32+\delta)}\mathcal U_{h_{K,\rho}}f}_{L^p(\R^2)}
\leq C\,2^{-j(\frac32+\delta)}K^{-1/p}\norm{f}_{L^p(\R^2)}
\end{equation}
uniformly in $\rho,j,K$. Set
\[
h_K=H-\big(\frac12+\sum_{\substack{0<|k|\leq K\\k\text{ odd}}}\frac{e^{ikt}}{\pi ik}\big)=\sum_{\substack{|k|>K\\k\text{ odd}}}\frac{e^{ikt}}{\pi ik}.
\]
Then $h_{K,\rho}\to h_K$ in $L^{p}(\T)$ as $\rho\to 1$. Indeed, this follows from $H_\rho\to H$ in $L^{p}(\T)$ as $\rho\to 1$ and
convergence of the finitely many low coefficients:
$$\|h_{K,\rho}-h_K\|_{L^p(\T)}\leq \|H_\rho-H\|_{L^p(\T)}+\sum_{\substack{|k|\leq K\\k\text{ odd}}}\Big\|(\rho^{|k|}-1)\frac{e^{ikt}}{\pi ik}\Big\|_{L^p(\T)}\to 0.$$
Estimate \eqref{eq:3} makes
the corresponding operators converge in norm, so \eqref{eq:8} passes to $h_K$:
\begin{equation}\label{eq:8'}
\norm{2^{-j(\frac32+\delta)}\mathcal U_{h_{K}}f}_{L^p(\R^2)}
\leq C\,2^{-j(\frac32+\delta)}K^{-1/p}\norm{f}_{L^p(\R^2)}.
\end{equation}
Recall 
\[
H-h_K=\frac12+\sum_{\substack{0<|k|\leq K\\k\text{ odd}}}\frac{e^{ikt}}{\pi ik}
\]
The nonzero low modes have coefficients $(\pi i k)^{-1}$, and the constant mode
is $1/2$. Lemma~\ref{lem:low} gives $C\,2^{-2j/p}$ for each mode, uniformly in
$k$. Hence
\begin{equation}\label{eq:sum-low}
\begin{aligned}
\norm{2^{-j(\frac32+\delta)}\mathcal U_{(H-h_{K})}f}_{L^p(\R^2)}
&\leq C\,2^{-j(\frac32+\delta)}2^{-2j/p}
\Bigl(\frac12+\sum_{1\leq|k|\leq K}\frac1{|k|}\Bigr)\norm{f}_{L^p(\R^2)}\\
&\leq C\,2^{-j(\frac32+\delta)}2^{-2j/p}(1+\log K)\norm{f}_{L^p(\R^2)}.
\end{aligned}
\end{equation}
Taking $K=2^{2j}$, \eqref{eq:8'} and \eqref{eq:sum-low} show
\begin{equation}
\norm{2^{-j(\frac32+\delta)}\mathcal U_{H}f}_{L^p(\R^2)}
\leq C\,2^{-j(\frac32+\delta)}2^{(-\frac{2}{p}+\epsilon')j}\norm{f}_{L^p(\R^2)},
\end{equation}
where $0<\epsilon' \ll 1$. By \eqref{peri} and the definition of $\mathcal U_H$ \eqref{eq:Uh}, we obtain
$$\big\|\tilde{T}_{G}^j f\big\|_{L^p(\R^2)}\leq 2^{-j(\frac32+\delta)} 2^{(-\frac{2}{p}+\epsilon')j}\|f\|_{L^{p}(\R^2)},$$
which gives \eqref{est:TGj'4}.

\end{proof}

 \medskip

\section{The proof of Proposition \ref{prop:TDj} }\label{sec:TDj}

In this section, we prove Proposition \ref{prop:TDj}, which is the second key proposition in the proof of the main theorem. We prove this proposition by considering the cases $j=0$ and 
$j\geq 1$. If $j=0$, from \eqref{equ:ream1}, \eqref{equ:ream2} and \eqref{equ:ream3}, by an argument similar to the proof for $T_{D_2}$, we have
$$|K_D^{\ell,0}(r_1,r_2;\theta_1-\theta_2)|\leq \beta_0(r_1+r_2)\in L^1 (\R).$$
Consequently, for $1\leq p\leq+\infty$, $$ \|T_{D}^{\ell,0}\|_{L^p(\R^2)\to L^p(\R^2)}\leq C,$$
which yields \eqref{est:TDj} for $j=0$.
\vspace{0.1cm}

From now on, we assume $j\geq1$.  For $j\geq1$ and $\ell=1,2,3$, we define 
$$\tilde{K}_D^{\ell,j}(r_1,r_2;\theta_1-\theta_2)=K_D^{\ell,{j}}(2^jr_1,2^jr_2;\theta_1-\theta_2)=\beta(r_1+r_2)
K^\ell_D(2^jr_1,2^jr_2;\theta_1-\theta_2)$$
and
\begin{equation}
\tilde{T}_{D}^{\ell,j} f(r_1,\theta_1)=\int_0^\infty\int_0^{2\pi} \tilde{K}_D^{\ell,j}(r_1,r_2;\theta_1-\theta_2) f(r_2,\theta_2) \,r_2dr_2\,d\theta_2.
\end{equation}
Then, using scaling, we have
$$T_{D}^{\ell,j} f(r_1,\theta_1)=2^{2j}\tilde{T}_{D}^{\ell,j} \big(f(2^{j}r_2,\theta_2)\big)(2^{-j}r_1,\theta_1).$$
Therefore, the inequalities \eqref{est:TDj} are equivalent to
 \begin{equation}\label{est:TDjequ}
 \begin{split}
 \|\tilde{T}_{D}^{\ell,j}\|_{L^p\to L^p}&\leq C2^{-j(\frac32+\delta)} 2^{-\frac{2j}{p}}\|f\|_{L^{p}(\R^2)}, \quad p\geq6,\\
  \|\tilde{T}_{D}^{\ell,j}\|_{L^p\to L^p}&\leq C2^{-j(\frac32+\delta)} 2^{-\frac{2j}{p}}2^{j\epsilon'}\|f\|_{L^{p}(\R^2)}, \quad 4<p<6,\\
  \|\tilde{T}_{D}^{\ell,j}\|_{L^2\to L^2}& \leq C2^{-j(\frac32+\delta)} 2^{-\frac{j}{2}}\|f\|_{L^{2}(\R^2)}.
 \end{split}
\end{equation}
To this end, we need the following fundamental proposition.
\begin{proposition}\label{lem:kerlplq}
Let $T_K$ be defined by
$$T_Kf(r_1,\theta_1):=\int_0^\infty \int_{0}^{2\pi} K(j; r_1,r_2,\theta_1,\theta_2)f(r_2,\theta_2)\;d\theta_2\;r_2\;dr_2,$$
and assume that the kernel $K(r_1,r_2,\theta_1,\theta_2)$ satisfies
\begin{equation}\label{equ:kercond}
  |K(j; r_1,r_2,\theta_1,\theta_2)|\lesssim 2^{-j(\frac32+\delta)} \big(1+2^jr_1r_2\big)^{-\frac12}\beta(r_1+r_2), \, \beta\in\mathcal{C}_c^\infty\big(\big[\tfrac38,\tfrac43\big]\big).
\end{equation}
Then, for $p>4$, there holds
\begin{equation}\label{equ:tkpqest}
  \big\|T_Kf\big\|_{L^p(\R^2)}\lesssim 2^{-j(\frac32+\delta)} 2^{-\frac2pj}\|f\|_{L^p(\R^2)},
\end{equation}
and
\begin{equation}\label{equ:tkpqest'}
  \big\|T_Kf\big\|_{L^2(\R^2)}\lesssim 2^{-j(\frac32+\delta)} 2^{-\frac j2}\|f\|_{L^2(\R^2)}.
\end{equation}
\end{proposition}

\begin{proof} For simplicity, in the proof we identify $f(r_2,\theta_2)$ with $f(y)=f(r_2\cos\theta_2,r_2\sin\theta_2)$.
  By \eqref{equ:kercond} and Minkowski's inequality, we obtain
  \begin{align*}
    \big\|T_Kf\big\|_{L^p(\R^2)}\lesssim & 2^{-j(\frac32+\delta)} \Big\|\int_0^\infty \int_{0}^{2\pi} \big(1+2^jr_1r_2\big)^{-\frac12}\beta(r_1+r_2) |f(r_2,\theta_2)|
    \;d\theta_2\;r_2\;dr_2\Big\|_{L^p(\R^2)}\\
    \lesssim& 2^{-j(\frac32+\delta)} \int_{|y|\leq 1}\big\|(1+2^j|x|\cdot|y|)^{-\frac12}\big\|_{L^p_x(|x|\leq 2)}|f(y)|\;dy\\
     \lesssim& 2^{-j(\frac32+\delta)} 2^{-\frac2pj}\int_{|y|\leq 2}\big\|(1+|x|)^{-\frac12}\big\|_{L^p_x(\R^2)}|y|^{-\frac2p}|f(y)|\;dy\\
     \lesssim&2^{-j(\frac32+\delta)} 2^{-\frac2pj}\|f\|_{L^p}\big\||y|^{-\frac2p}\big\|_{L^{p'}(|y|\leq 1)}
     \lesssim 2^{-j(\frac32+\delta)} 2^{-\frac2pj}\|f\|_{L^p},
  \end{align*}
  where we need the assumption $p>4$ to guarantee that $(1+|x|)^{-1/2}\in L^p(\R^2)$ and $|y|^{-\frac2p}\in L^{p'}(|y|\leq 2).$ For $p=2$, since $|x|^{-\frac12}\in L^2(|x|\leq 2)$, it is easy to show that
  \begin{align*}
    \big\|T_Kf\big\|_{L^2(\R^2)}\lesssim & 2^{-j(\frac32+\delta)} 2^{-\frac j2}\Big\|r_1^{-\frac12}\int_0^2 \int_{0}^{2\pi} r_2^{-\frac12}\beta(r_1+r_2) |f(r_2,\theta_2)|
    \;d\theta_2\;r_2\;dr_2\Big\|_{L^2(\R^2)}\\
    \lesssim& 2^{-j(\frac32+\delta)} 2^{-\frac j2} \big\||x|^{-\frac12}\big\|_{L^2_x(|x|\leq 2)}\int_{|y|\leq 2}|y|^{-\frac12} |f(y)|\;dy\\
     \lesssim &2^{-j(\frac32+\delta)} 2^{-\frac j2}\|f\|_{L^2(\R^2)}.
  \end{align*}

\end{proof}

Now, we show that the kernels $\tilde{K}_D^{\ell,j}(r_1,r_2;\theta_1-\theta_2)$ satisfy \eqref{equ:kercond}. To do this, for fixed $r_1, r_2$, $\theta_1,\theta_2$, and $j\geq1$, we define
\begin{equation}
\phi(r_1,r_2; s)=|{\bf n}_s|=\sqrt{r_1^2+r_2^2+2r_1r_2\cosh s},
\end{equation}
and let $\alpha\in(-1,1)\setminus\{0\}$, $a=a_\pm$ in \eqref{ker-D}, and
\begin{equation}
\begin{split}
\psi_1(r_1,r_2,\theta_1,\theta_2; s)&=(2^{-j}+|{\bf n}_s|)^{-\frac32-\delta}  a(2^j|{\bf n}_s|)\, e^{-|\alpha|s},\\
\psi_2(r_1,r_2,\theta_1,\theta_2; s)&=(2^{-j}+|{\bf n}_s|)^{-\frac32-\delta}  a(2^j|{\bf n}_s|)\, \frac{(e^{-s}-\cos(\theta_1-\theta_2+\pi))\sinh(\alpha s)}{\cosh(s)-\cos(\theta_1-\theta_2+\pi)},\\
\psi_3(r_1,r_2,\theta_1,\theta_2; s)&=(2^{-j}+|{\bf n}_s|)^{-\frac32-\delta}  a(2^j|{\bf n}_s|)\, \frac{\sin(\theta_1-\theta_2+\pi)\cosh(\alpha s)}{\cosh(s)-\cos(\theta_1-\theta_2+\pi)}.
\end{split}
\end{equation}
Then by definition, we have
\begin{equation}\label{KDl}
\begin{split}
\tilde{K}_D^{\ell,j}(r_1,r_2;\theta_1-\theta_2)&=2^{-j(\frac32+\delta)} \beta(r_1+r_2) \int_0^\infty e^{i2^j\phi(s)} \psi_\ell(s)\;ds,\quad \ell=1,2,3.
\end{split}
\end{equation}

\begin{lemma}\label{lem:ker-est12} For  $j\geq1$, there holds
\begin{equation}\label{kerD-est}
\begin{split}
|\tilde{K}_{D}^{1,j}(r_1,r_2;\theta_1-\theta_2)|+|\tilde{K}_{D}^{2,j}(r_1,r_2;\theta_1-\theta_2)|\lesssim 2^{-j(\frac32+\delta)} \big(1+2^jr_1r_2\big)^{-\frac12}.
\end{split}
\end{equation}
\end{lemma}
{The kernel} $\tilde{K}_D^{3,j}$ is more complicated, as it is more singular at $s=0$ than $\tilde{K}_D^{\ell,j}$ with $\ell=1, 2$.
To treat $\tilde{K}_D^{3,j}$, we define
\begin{equation}\label{psi-m}
\begin{split}
\psi_{3,m}(r_1,r_2,\theta_1,\theta_2; s)&=(r_1+r_2)^{-\frac32-\delta} a(2^j(r_1+r_2))\, \frac{\sin(\theta_1-\theta_2+\pi)}{\frac{s^2}2+2\sin^2\big(\tfrac{\theta_1-\theta_2+\pi}2\big)},
\end{split}
\end{equation}
and $\psi_{3,e}=\psi_{3}-\psi_{3,m}$.
We further define the kernels
\begin{equation}\label{KDe3}
\begin{split}
\tilde{K}_{D,m}^{3,j}(r_1,r_2;\theta_1-\theta_2)&=2^{-j(\frac32+\delta)}\beta(r_1+r_2) \int_0^\infty e^{i2^j\phi(s)} \psi_{3, m}(s)\;ds,\\
\tilde{K}_{D,e}^{3,j}(r_1,r_2;\theta_1-\theta_2)&=2^{-j(\frac32+\delta)}\beta(r_1+r_2) \int_0^\infty e^{i2^j\phi(s)} \psi_{3, e}(s)\;ds.
\end{split}
\end{equation}
Similarly to Lemma \ref{lem:ker-est12}, we have

\begin{lemma}\label{lem:ker-est3} Let $\tilde{K}_{D,e}^{3,j}$ be given in \eqref{KDe3} and define
\begin{equation}\label{def-H}
H(r_1,r_2;\theta_1-\theta_2)=2^{-j(\frac32+\delta)}\beta(r_1+r_2) \int_0^\infty e^{i 2^j\frac{r_1r_2}{r_1+r_2} s^2} \psi_{3, m}( s)\;ds,
\end{equation}
 then
\begin{equation}\label{kerDe3-est}
\begin{split}
|\tilde{K}_{D,e}^{3,j}(r_1,r_2;\theta_1-\theta_2)|\lesssim 2^{-j(\frac32+\delta)}\big(1+2^jr_1r_2\big)^{-\frac12},
\end{split}
\end{equation}
and
\begin{equation}\label{kerDm3-est}
\begin{split}
\big|e^{-i2^j(r_1+r_2)}\tilde{K}_{D,m}^{3,j}(r_1,r_2;\theta_1-\theta_2)-& H(r_1,r_2;\theta_1-\theta_2)\big|
\\&\lesssim 2^{-j(\frac32+\delta)}\big(1+2^jr_1r_2\big)^{-\frac12}.
\end{split}
\end{equation}

\end{lemma}

\begin{proposition}\label{prop:ker-est3H} Let $b=\sqrt{2}\sin\big(\frac{\theta_1-\theta_2+\pi}2\big)$ and $H(r_1,r_2;\theta_1-\theta_2)$ be {as} in \eqref{def-H} and let $\chi\in\mathcal{C}^\infty_c([-\epsilon,\epsilon])$ be an even bump function supported in a small interval $[-\epsilon,\epsilon]$ where $0<\epsilon\ll1$.
Then
\begin{equation}\label{kerDH3-est}
\begin{split}
|(1-\chi(b))H(r_1,r_2;\theta_1-\theta_2)|\lesssim 2^{-j(\frac32+\delta)}\big(1+2^jr_1r_2\big)^{-\frac12}.
\end{split}
\end{equation}
Furthermore, it holds that
\begin{equation}\label{equ:Hpqest'}
\begin{split}
  \Big\|\int_0^\infty \int_{0}^{2\pi} \chi(b)H(r_1,r_2;\theta_1-\theta_2)& f(r_2,\theta_2)\;d\theta_2\;r_2\;dr_2\Big\|_{L^2(\R^2)}
  \\&\lesssim 2^{-j(\frac32+\delta)}2^{-\frac j2}\|f\|_{L^2(\R^2)},
    \end{split}
\end{equation}
and for $p\geq 6$,
\begin{equation}\label{equ:Hpqest}
\begin{split}
  \Big\|\int_0^\infty \int_{0}^{2\pi} \chi(b)H(r_1,r_2;\theta_1-\theta_2) &f(r_2,\theta_2)\;d\theta_2\;r_2\;dr_2\Big\|_{L^p(\R^2)}
  \\&\lesssim 2^{-j(\frac32+\delta)}2^{-j\frac 2p}\|f\|_{L^p(\R^2)},
  \end{split}
\end{equation}
while for $4< p<6$ and $0<\epsilon'\ll 1$
\begin{equation}\label{equ:Hpqest''}
\begin{split}
  \Big\|\int_0^\infty \int_{0}^{2\pi} \chi(b)H(r_1,r_2;\theta_1-\theta_2) &f(r_2,\theta_2)\;d\theta_2\;r_2\;dr_2\Big\|_{L^p(\R^2)}
  \\&\lesssim 2^{-j(\frac32+\delta)}2^{-j(\frac 2p-\epsilon')}\|f\|_{L^p(\R^2)}.
  \end{split}
\end{equation}

\end{proposition}

\begin{remark}Compared with \eqref{equ:Hpqest}, there is an arbitrarily  small loss in \eqref{equ:Hpqest''}. However, this is sufficient for our main {theorem} \ref{thm:LA0}.
\end{remark}

By applying Lemma \ref{lem:ker-est12} and Proposition \ref{lem:kerlplq}, we obtain \eqref{est:TDjequ} for $\ell = 1, 2$. For $\ell = 3$, we again use Proposition \ref{lem:kerlplq}, together with Lemma \ref{lem:ker-est3} and \eqref{kerDH3-est}, to prove \eqref{est:TDjequ} when $b = \sqrt{2} \sin\big(\frac{\theta_1 - \theta_2 + \pi}{2}\big)$ is large. When $b$ is small, we complete the proof of \eqref{est:TDjequ} using \eqref{equ:Hpqest} and \eqref{equ:Hpqest''}. Therefore, we are reduced to proving Lemma \ref{lem:ker-est12} and Lemma \ref{lem:ker-est3}, whose proofs are given in the appendix.

\begin{proof}[{\bf The proof of Proposition \ref{prop:ker-est3H}:}] 
We first prove \eqref{kerDH3-est}.
Recall that 
$b=\sqrt{2}\sin\big(\frac{\theta_1-\theta_2+\pi}2\big)$ and the definitions \eqref{def-H} and \eqref{psi-m}.  By scaling, it suffices to show
\begin{equation}\label{kerDH3-est'}
\Big|\beta(r_1+r_2)  a(2^j(r_1+r_2)) (1-\chi(b)) \int_0^\infty e^{i2^{j+1}\frac{r_1r_2}{r_1+r_2} s^2} \, \frac{b}{s^2+b^2}
\;ds\Big|\lesssim \big(1+2^jr_1r_2\big)^{-\frac12}.
\end{equation}
If $s\geq1$, the phase function satisfies that
$$\Big|\partial_s \big(2^{j+1}\frac{r_1r_2}{r_1+r_2} s^2\big)\Big|\geq 2^{j}r_1r_2,$$
hence the integration by parts shows 
\begin{equation*}
\Big|\beta(r_1+r_2)  a(2^j(r_1+r_2)) (1-\chi(b)) \int_1^\infty e^{i2^{j+1}\frac{r_1r_2}{r_1+r_2} s^2} \, \frac{b}{s^2+b^2}
\;ds\Big|\lesssim \big(1+2^jr_1r_2\big)^{-1}.
\end{equation*}
Since outside of the support of $\chi$, i.e., $|b|\geq \epsilon$, we have
$$\int_0^1\Big|\partial_s\Big(\frac{b}{s^2+b^2}\Big)\Big| ds\leq \frac C{|b|}\leq C_\epsilon.$$
By using the Van der Corput Lemma \ref{lem:VCL}, we have
\begin{equation*}
\Big|\beta(r_1+r_2)  a(2^j(r_1+r_2)) (1-\chi(b)) \int_0^1 e^{i2^{j+1}\frac{r_1r_2}{r_1+r_2} s^2} \, \frac{b}{s^2+b^2}
\;ds\Big|\lesssim \big(1+2^jr_1r_2\big)^{-\frac12}.
\end{equation*}
Therefore we prove \eqref{kerDH3-est'}.\vspace{0.1cm}

We next prove \eqref{equ:Hpqest}. The main difficulty arises when $b \to 0$. As $b \to 0$, due to the periodicity of the sine function, we have $\frac12 (\theta_1 - \theta_2 + \pi) \to 0$, $\pi$, or $2\pi$. Without loss of generality, to prove \eqref{equ:Hpqest}, we only consider the case $\theta := \frac12 (\theta_1 - \theta_2 + \pi) \to 0$, since the other cases can be handled by replacing $\theta_2$ with $\theta_2 \pm 2\pi$.

Let $\theta=\frac12 (\theta_1-\theta_2+\pi)$,  and define 
\begin{equation}\label{def:tH}
\begin{split}
\tilde{H}=2^{-j(\frac 32+\delta)}\chi(\theta) &(r_1+r_2)^{-\frac32-\delta}\beta(r_1+r_2) \\
&\times a(2^j(r_1+r_2))\int_0^\infty e^{i2^{j+1}\frac{r_1r_2}{r_1+r_2} s^2} \, \frac{\sqrt{2}\theta}{s^2+2\theta^2}
\;ds.
\end{split}
\end{equation}
As $\theta\to0$, {we have} $\sin\theta\sim \theta$,  {and} one can verify that
$$\int_0^1\Big|\partial_s\Big(\frac{\sqrt{2}\sin\theta}{s^2+2\sin^2\theta}-\frac{\sqrt{2}\theta}{s^2+2\theta^2}\Big)\Big| ds\leq  C,\quad \theta\to 0.$$
The same argument as before, using {the} Van der Corput Lemma \ref{lem:VCL}, shows
$$\big|\chi(b)H(r_1,r_2;\theta_1-\theta_2)-\tilde{H}(r_1,r_2;\theta)\big|\lesssim 2^{-j(\frac 32+\delta)}\big(1+2^jr_1r_2\big)^{-\frac12}.$$
Therefore, it suffices to show \eqref{equ:Hpqest} and \eqref{equ:Hpqest'} with $\chi(b)H$ replaced by $\tilde{H}$. We first prove \eqref{equ:Hpqest'}. For $p\geq 2$, we write
\begin{equation}\label{est:tri}
\begin{split}
  &\Big\|\int_0^\infty \int_{0}^{2\pi} \tilde{H}(r_1,r_2;\theta_1-\theta_2) f(r_2,\theta_2)\;d\theta_2\;r_2\;dr_2\Big\|_{L^p(\R^2)}
  \\&\lesssim  \Big\|\int_0^\infty \Big\|\int_{\R} \tilde{H}(r_1,r_2;\theta_1-\theta_2) \mathbbm{1}_{[0,2\pi]}(\theta_2) f(r_2,\theta_2)\;d\theta_2\Big\|_{L^p_{\theta_1}(\R)} \;r_2\;dr_2\Big\|_{L^p_{r_1dr_1}([0,+\infty)}.
  \end{split}
\end{equation}
Using the convolution kernel and Plancherel's theorem, for $p = 2$, we obtain
\begin{equation*}
\begin{split}
& \Big\|\int_{\R} \tilde{H}(r_1,r_2;\theta_1-\theta_2) \mathbbm{1}_{[0,2\pi]}(\theta_2) f(r_2,\theta_2)\;d\theta_2\Big\|_{L^2_{\theta_1}(\R)} 
 \\&\lesssim \|\hat{\tilde{H}}(r_1,r_2;\zeta)\|_{L^\infty_{\zeta}(\R)} \big\| \widehat{f\mathbbm{1}_{[0,2\pi]} }(r_2,\zeta)\big\|_{L_{\zeta}^{2}(\R)}.
  \end{split}
\end{equation*}

\begin{lemma} Let $\tilde{H}(r_1,r_2;\theta) $ be defined in \eqref{def:tH}. Then the Fourier transform of $\tilde{H}(r_1,r_2;\theta) $ satisfies the bounds
\begin{equation}\label{FtH}
\big|\hat{\tilde{H}}(r_1,r_2;\zeta)  \big|\lesssim 2^{-j(\frac 32+\delta)}\times
\begin{cases} (2^j r_1r_2)^{-\frac12}, \quad &|\zeta|\lesssim (2^j r_1r_2)^{\frac12},\\
\frac1{|\zeta|}, \quad &|\zeta|\gg (2^j r_1r_2)^{\frac12}.
\end{cases}
\end{equation}
\end{lemma}
\begin{proof}
The function $\tilde{H}(r_1,r_2;\theta)$ given in \eqref{def:tH} is similar to $\tilde{H}$ in \cite[Lemma 2.3]{BFM18}. One can prove this lemma using the conjugate Poisson kernel; therefore, we omit the details.
\end{proof}

Therefore, using the fact that $r_1 + r_2 \sim 1$, we finally obtain \eqref{equ:Hpqest'}:
\begin{equation*}
\begin{split}
  &\Big\|\int_0^\infty \int_{0}^{2\pi} \tilde{H}(r_1,r_2;\theta_1-\theta_2) f(r_2,\theta_2)\;d\theta_2\;r_2\;dr_2\Big\|_{L^2(\R^2)}
  \\&\lesssim 2^{-j(\frac 32+\delta)} 2^{-\frac j2} \Big\|\int_0^1 (r_1r_2)^{-\frac12}\big\| f(r_2,\theta_2)\big\|_{L_{\theta_2}^{2}([0,2\pi])} \;r_2\;dr_2\Big\|_{L^2_{r_1dr_1}([0,1])}
    \\&\lesssim 2^{-j(\frac 32+\delta)} 2^{-\frac j2} \|f\|_{L^2(\R^2)}.
  \end{split}
\end{equation*}

Now we prove \eqref{equ:Hpqest}. Note that $r_1 + r_2 \sim 1$ due to the compact support of $\beta$; hence $f$ has compact support.
Recalling the partition of unity
$$\beta_0(r)=1-\sum_{j\geq1}\beta_j(r),\quad \beta_j(r)=\beta(2^{-j}r),\quad\beta\in\mathcal{C}_c^\infty\big(\big[\tfrac38,\tfrac43\big]\big),$$
we define $\tilde{H}_\ell$ by
\begin{equation}
\widehat{\tilde{H}_\ell}(r_1,r_2,\zeta)=\beta_{\ell}(|\zeta|) \hat{\tilde{H}}(r_1,r_2;\zeta).
\end{equation}
Thus, using the Bernstein inequality and \eqref{FtH}, we have
\begin{equation}\label{FtH'}
\begin{split}
\Big\|\tilde{H}_\ell (r_1,r_2;\cdot)\ast f(\theta_1)\Big\|_{L^p_{\theta_1}(\R)}& \lesssim 2^{-j(\frac 32+\delta)}\\
&\times
\begin{cases} (2^j r_1r_2)^{-\frac12}2^{\ell(\frac12-\frac1p)}\|f_\ell\|_{L^2_{\theta_2}}, \quad &2^\ell \lesssim (2^j r_1r_2)^{\frac12},\\
2^{-\ell(\frac12+\frac1p)}\|f_\ell\|_{L^2_{\theta_2}}, \quad &2^{\ell}\gg (2^j r_1r_2)^{\frac12},
\end{cases}
\end{split}
\end{equation}
where $f_\ell$ is defined by $\widehat{ f_\ell}=\beta_\ell(\zeta) \hat{f}(\zeta)$.
Then, by using the Littlewood-Paley square function estimates, the Minkowski inequality, and \eqref{FtH'}, we have 
\begin{equation}\label{key1}
\begin{split}
& \Big\|\int_{\R} \tilde{H}(r_1,r_2;\theta_1-\theta_2) \mathbbm{1}_{[0,2\pi]}(\theta_2) f(r_2,\theta_2)\;d\theta_2\Big\|^2_{L^p_{\theta_1}(\R)} \\
&= \Big\|\sum_{\ell=0}^\infty \tilde{H}_\ell (r_1,r_2;\cdot)\ast (\mathbbm{1}_{[0,2\pi]}f)(\theta_1)\Big\|^2_{L^p_{\theta_1}(\R)} 
= \sum_{\ell=0}^\infty \Big\|\tilde{H}_\ell (r_1,r_2;\cdot)\ast (\mathbbm{1}_{[0,2\pi]}f)(\theta_1)\Big\|^2_{L^p_{\theta_1}(\R)} 
 \\&\lesssim 2^{-2j(\frac 32+\delta)}\Big( (2^j r_1r_2)^{-1}\sum_{2^\ell \lesssim (2^j r_1r_2)^{\frac12}}2^{\ell(1-\frac2p)}\|(\mathbbm{1}_{[0,2\pi]}f)_\ell\|^2_{L^2_{\theta_2}}
 \\&\qquad\qquad\qquad+ \sum_{2^\ell \gg (2^j r_1r_2)^{\frac12}}
 2^{-2\ell(\frac12+\frac1p)}\|(\mathbbm{1}_{[0,2\pi]}f)_\ell\|^2_{L^2_{\theta_2}}\Big).
  \end{split}
\end{equation}
To estimate \eqref{equ:Hpqest}, by using Minkowski's inequality, we obtain
\begin{equation}
\begin{split}
&\int_0^\infty \Big\|\int_{\R} \tilde{H}(r_1,r_2;\theta_1-\theta_2) \mathbbm{1}_{[0,2\pi]}(\theta_2) f(r_2,\theta_2)\;d\theta_2\Big\|_{L^p(d\theta_1r_1dr_1)} \;r_2\;dr_2\\
&\lesssim \int_0^1 \Big(\int_{0}^1 \big\|\tilde{H}(r_1,r_2;\cdot) \ast(\mathbbm{1}_{[0,2\pi]}f)(r_2,\theta_1)\big\|^p_{L^p(d\theta_1)} r_1dr_1\Big)^{1/p} \;r_2\;dr_2.
  \end{split}
\end{equation}
This together with \eqref{key1} yields
\begin{equation}
\begin{split}
&\int_0^\infty \Big\|\int_{\R} \tilde{H}(r_1,r_2;\theta_1-\theta_2) \mathbbm{1}_{[0,2\pi]}(\theta_2) f(r_2,\theta_2)\;d\theta_2\Big\|_{L^p(d\theta_1r_1dr_1)} \;r_2\;dr_2\\
&\lesssim \int_0^1 \Big(\int_{0}^1 \big\|\tilde{H}(r_1,r_2;\cdot) \ast(\mathbbm{1}_{[0,2\pi]}f)(r_2,\theta_1)\big\|^p_{L^p(d\theta_1)} r_1dr_1\Big)^{1/p} \;r_2\;dr_2,\\
&\lesssim 2^{-j(\frac 32+\delta)}  \int_0^1\Big[ \int_0^1 \Big( (2^j r_1r_2)^{-1}\sum_{2^\ell \lesssim (2^j r_1r_2)^{\frac12}}2^{\ell(1-\frac2p)}\|(\mathbbm{1}_{[0,2\pi]}f)_\ell\|^2_{L^2_{\theta_2}}
 \\&\qquad\qquad+ \sum_{2^\ell \gg (2^j r_1r_2)^{\frac12}}
 2^{-2\ell(\frac12+\frac1p)}\|(\mathbbm{1}_{[0,2\pi]}f)_\ell\|^2_{L^2_{\theta_2}}\Big)^{\frac p2} \;r_1\;dr_1\Big]^{1/p}\;r_2\;dr_2.
  \end{split}
\end{equation}
We use Minkowski's inequality to estimate the first term
\begin{equation}
\begin{split}
&\Big(\int_0^1 \Big((2^j r_1r_2)^{-1}\sum_{2^\ell \lesssim (2^j r_1r_2)^{\frac12}}2^{\ell(1-\frac2p)}\|(\mathbbm{1}_{[0,2\pi]}f)_\ell\|^2_{L^2_{\theta_2}}\Big)^{\frac p2}\;r_1\;dr_1\Big)^{\frac2p}\\
&\lesssim (2^j r_2)^{-1}\sum_{\ell\geq 0}2^{\ell(1-\frac2p)}\|(\mathbbm{1}_{[0,2\pi]}f)_\ell\|^2_{L^2_{\theta_2}}\;\Big(\int_{2^{2\ell}(2^jr_2)^{-1} \lesssim r_1}   r_1^{1-\frac p2}\;dr_1\Big)^{\frac2p}\\
&\lesssim (2^j r_2)^{-\frac4p}\sum_{\ell\geq 0}2^{2\ell(\frac3p-\frac12)}\|(\mathbbm{1}_{[0,2\pi]}f)_\ell\|^2_{L^2_{\theta_2}}.
  \end{split}
\end{equation}
For the second term, we similarly obtain
\begin{equation}
\begin{split}
&\Big(\int_0^1 \Big(\sum_{2^\ell \gg (2^j r_1r_2)^{\frac12}}
 2^{-2\ell(\frac12+\frac1p)}\|(\mathbbm{1}_{[0,2\pi]}f)_\ell\|^2_{L^2_{\theta_2}}\Big)^{\frac p2}\;r_1\;dr_1\Big)^{\frac2p}\\
&\lesssim \sum_{\ell\geq 0}2^{-2\ell(\frac12+\frac1p)}\|(\mathbbm{1}_{[0,2\pi]}f)_\ell\|^2_{L^2_{\theta_2}}\;\Big(\int_{ r_1\ll 2^{2\ell}(2^jr_2)^{-1} }   r_1\;dr_1\Big)^{\frac2p}\\
&\lesssim (2^jr_2)^{-\frac4p}\sum_{\ell\geq 0}2^{2\ell(\frac3p-\frac12)}\|(\mathbbm{1}_{[0,2\pi]}f)_\ell\|^2_{L^2_{\theta_2}}.
  \end{split}
\end{equation}
If $p\geq 6$, then we obtain
\begin{equation}
\begin{split}
&\int_0^\infty \Big\|\int_{\R} \tilde{H}(r_1,r_2;\theta_1-\theta_2) \mathbbm{1}_{[0,2\pi]}(\theta_2) f(r_2,\theta_2)\;d\theta_2\Big\|_{L^p_{\theta_1}(\R)} \;r_2\;dr_2\\
&\lesssim 2^{-j(\frac 32+\delta)} \int_0^1 (2^jr_2)^{-\frac2p}\|f\|_{L^p_{\theta_2}([0,2\pi))}\;r_2\;dr_2\\
&\lesssim 2^{-j(\frac 32+\delta)} 2^{-j\frac2p}.
  \end{split}
\end{equation}
This proves \eqref{equ:Hpqest}.\vspace{0.2cm}

We finally prove  \eqref{equ:Hpqest''}. For this purpose, we need the following lemma.
\begin{lemma} Define $I_j(r_1, r_2; \theta)$ to be the integral
\begin{equation}\label{def:I}
I_j(r_1, r_2; \theta)= \beta(r_1+r_2)\int_0^\infty e^{i2^{j+1}\frac{r_1r_2}{r_1+r_2} s^2} \, \frac{\sqrt{2}\theta}{s^2+2\theta^2}
\;ds.
\end{equation}
Then there exists a constant $C$ such that
\begin{equation}\label{est:I}
|I_j(r_1, r_2; \theta)|\leq C(1+2^j r_1r_2\theta^2)^{-\frac12}.
\end{equation}

\end{lemma}

\begin{proof} Using {the change of variables} $s\to \frac{s\theta}{\sqrt{2}}$, we have
\begin{equation*}
I_j(r_1, r_2; \theta)= \beta(r_1+r_2)\int_0^\infty e^{i2^{j}\frac{r_1r_2}{r_1+r_2} \theta^2 s^2} \, \frac{2}{s^2+4}
\;ds.
\end{equation*}
It is easy to see that $|I_j(r_1, r_2; \theta)|\leq C$. By applying Van der Corput Lemma \ref{lem:VCL}, and noting  that
$$\partial_s^2\big(2^{j}\frac{r_1r_2}{r_1+r_2} \theta^2 s^2\big)\geq 2^{j} r_1r_2 \theta^2,$$we have
\begin{equation*}
|I_j(r_1, r_2; \theta)|\leq C (2^j r_1r_2\theta^2)^{-\frac12}.
\end{equation*}

\end{proof}
Recall the definition of  $\tilde{H}$ in \eqref{def:tH},
\begin{equation}
\begin{split}
\tilde{H}&=2^{-j(\frac 32+\delta)}\chi(\theta) (r_1+r_2)^{-\frac32-\delta} a(2^j(r_1+r_2))I_j(r_1, r_2; \theta),
\end{split}
\end{equation}
we obtain, for any $0<\epsilon'\ll 1$,
\begin{equation}
\begin{split}
\int_0^{2\pi} |\tilde{H}(r_1,r_2;\theta)|\, d\theta &\lesssim  2^{-j(\frac 32+\delta)}\int_0^\epsilon \big(1+2^jr_1r_2\theta^2\big)^{-\frac12}\, d\theta
\\&\lesssim 2^{-j(\frac 32+\delta)} \big(1+2^jr_1r_2\big)^{-\frac12+\epsilon'}.
\end{split}
\end{equation}
Using {Young's} inequality, we obtain
\begin{equation}\label{equ:Hpqest4}
\begin{split}
&\int_0^\infty \Big\|\int_{\R} \tilde{H}(r_1,r_2;\theta_1-\theta_2) \mathbbm{1}_{[0,2\pi]}(\theta_2) f(r_2,\theta_2)\;d\theta_2\Big\|_{L^4(d\theta_1r_1dr_1)} \;r_2\;dr_2\\
&\lesssim \int_0^1 \Big(\int_{0}^1 \big\|\tilde{H}(r_1,r_2;\cdot) \ast(\mathbbm{1}_{[0,2\pi]}f)(r_2,\theta_1)\big\|^4_{L^4(d\theta_1)} r_1dr_1\Big)^{1/4} \;r_2\;dr_2\\
&\lesssim 2^{-j(\frac 32+\delta)} \int_0^1 \Big(\int_{0}^1  \big(1+2^jr_1r_2\big)^{-2+\epsilon'} r_1dr_1\Big)^{1/4} \big\|(\mathbbm{1}_{[0,2\pi]}f)(r_2,\theta_2)\big\|_{L^4(d\theta_2)} \;r_2\;dr_2\\
&\lesssim 2^{-j(\frac 32+\delta)} \int_0^1 \big(1+2^jr_2\big)^{-\frac12+\epsilon'} \big\|f(r_2,\theta_2)\big\|_{L^4_{\theta_2}([0,2\pi])} \;r_2\;dr_2\\
&\lesssim 2^{-j(\frac 32+\delta)} 2^{-j(\frac12-\epsilon')}\big\|f\big\|_{L^4(\R^2)} .
  \end{split}
\end{equation}
Therefore, \eqref{equ:Hpqest''} follows by interpolating with \eqref{equ:Hpqest4} and  \eqref{equ:Hpqest}.
\end{proof}

\medskip

\section{Appendix: The proofs of technical lemmas}

In this section, we provide  the proofs of   
Lemmas \ref{lem:ker-est12} and \ref{lem:ker-est3},
  which are required for the proofs of 
  Proposition \ref{prop:TDj}.
The proofs of the lemmas are modified from \cite[Lemma 4.4, Lemma 4.5, Lemma 4.6 ]{FZZ1}, but we  provide the details here for {the reader's} convenience.
This is {because} the phase function
\begin{equation}
\phi(r_1,r_2; s)=|{\bf n}_s|=\sqrt{r_1^2+r_2^2+2r_1r_2\cosh s},
\end{equation}
is {the same as} the phase function of the oscillatory integral in \cite{FZZ1},
and the only minor difference lies in the amplitude functions, where we  replace $|{\bf n}_s|^{-\frac12}$ by $(2^{-j}+|{\bf n}_s|)^{-\frac32-\delta}$. 

\begin{proof}[{\bf The proof of Lemma \ref{lem:ker-est12}}]: 
If $2^jr_1r_2\lesssim1$, {then, as before,} \eqref{kerD-est} follows by
using \eqref{equ:ream1} and \eqref{equ:ream2}.
So from now on, we assume $2^jr_1r_2\geq 1$ in the proof.
We first compute that
\begin{equation}
\begin{split}
\partial_s \phi&=\frac{r_1r_2\sinh s}{(r_1^2+r_2^2+2r_1r_2\cosh s)^{1/2}},\\
\partial^2_s\phi&=\frac{r_1r_2\cosh s}{(r_1^2+r_2^2+2r_1r_2\cosh s)^{1/2}}-\frac{(r_1r_2\sinh s)^2}{(r_1^2+r_2^2+2r_1r_2\cosh s)^{3/2}},
\end{split}
\end{equation}
therefore, thanks to the fact {that} $r_1+r_2\sim 1$, {which follows from} the support of $\beta$, we obtain, for $0\leq s\leq 1$
\begin{equation}
\partial_s \phi (0)=0,\quad |\partial^2_s \phi|\geq c \frac{r_1r_2}{r_1+r_2} \gtrsim r_1r_2,
\end{equation}
and for $s\geq 1$
\begin{equation}
\partial_s \phi\geq c \frac{r_1r_2}{r_1+r_2} \gtrsim r_1r_2.
\end{equation}
Therefore, we divide the integrals \eqref{kerD12} into two parts: $s\in[0,1]$ and $s\in [1,+\infty)$.

One can verify that $\partial_s\phi(s)$ is monotonic on the interval $[1,\infty)$ and {the fact} that
\begin{equation}\label{psi-b-1}
\int_1^\infty |\psi_1'(s)|ds+\int_1^\infty|\psi_2'(s)|ds\lesssim 1.
\end{equation} Indeed, if the derivative hits  $(2^{-j}+|{\bf n}_s|)^{-\frac32-\delta} a(2^j|{\bf n}_s|)$ which is bounded, we again use \eqref{equ:ream1} and \eqref{equ:ream2}  to obtain \eqref{psi-b-1}.
If the derivative hits  $e^{-|\alpha|s}$ or \begin{equation*}
\begin{split}
 \frac{(e^{-s}-\cos(\theta_1-\theta_2+\pi))\sinh(\alpha s)}{\cosh(s)-\cos(\theta_1-\theta_2+\pi)},
\end{split}
\end{equation*} the terms are harmless to obtain \eqref{psi-b-1} when $s\in[1,\infty)$.  By using the Van der Corput Lemma \ref{lem:VCL} with $k=1$, we can obtain the bounds $O(2^{-j(\frac 32+\delta)}\big(2^jr_1r_2\big)^{-1})$ if we replace the $\int_0^\infty$ of \eqref{kerD12}  by $\int_1^\infty$.
By applying the Van der Corput Lemma \ref{lem:VCL} with $k=2$ to $\int_0^1$ part, thus we prove
\begin{equation}
\begin{split}
|\tilde{K}_{D}^{1,j}(r_1,r_2;\theta_1-\theta_2)|\lesssim 2^{-j(\frac 32+\delta)}\big(2^jr_1r_2\big)^{-\frac12}.
\end{split}
\end{equation}
However, when we consider $\tilde{K}_{D}^{2,j}$, one need consider the derivative hits
\begin{equation*}
\begin{split}
 \frac{(e^{-s}-\cos(\theta_1-\theta_2+\pi))\sinh(\alpha s)}{\cosh(s)-\cos(\theta_1-\theta_2+\pi)},
\end{split}
\end{equation*}
which becomes singular near $s=0$
so that we {cannot} verify
\begin{equation*}
\int_0^1|\psi_2'(s)|ds\lesssim 1.
\end{equation*}
To finish the proof of Lemma \ref{lem:ker-est12}, it remains  to bound
\begin{equation}\label{psi-2'}
\begin{split}
2^{-j(\frac 32+\delta)}\beta(r_1+r_2) \int_0^1 e^{i2^j\phi(s)} \psi_2(s)\;ds.
\end{split}
\end{equation}
To this aim, when $s$ is close to $0$, we replace
\begin{equation*}
\begin{split}
 \frac{(e^{-s}-\cos(\theta_1-\theta_2+\pi))\sinh(\alpha s)}{\cosh(s)-\cos(\theta_1-\theta_2+\pi)}\sim \frac{(-s+b^2)(\alpha s)}{\frac{s^2}2+b^2}
\end{split}
\end{equation*}
where $b=\sqrt{2}\sin\big(\frac{\theta_1-\theta_2+\pi}2\big)$.
Then for $0\leq s\leq 1$,  uniformly in $b$, we have
\begin{equation}\label{d-v1}
\begin{split}
\Big|\partial_s^{k}\Big(\frac{(e^{-s}-\cos(\theta_1-\theta_2+\pi))\sinh(\alpha s)}{\cosh(s)-\cos(\theta_1-\theta_2+\pi)}-\frac{(-s+b^2)(\alpha s)}{\frac {s^2}2+b^2}\Big)\Big|\lesssim 1,\quad k=0,1
\end{split}
\end{equation}
which has been verified in \cite[Lemma 5.1]{FZZ1}.

Therefore  the difference term can be handled by applying Van der Corput Lemma \ref{lem:VCL} again. Hence, instead of \eqref{psi-2'}, we need to control
\begin{equation}\label{psi-2'1}
\begin{split}
2^{-j(\frac 32+\delta)}\beta(r_1+r_2) \int_0^1 e^{i2^j\phi(s)}(2^{-j}+|{\bf n}_s|)^{-\frac32-\delta} a(2^j|{\bf n}_s|)\frac{(-s+b^2)(\alpha s)}{\frac{s^2}2+b^2}\;ds,
\end{split}
\end{equation}
which is bounded by
\begin{equation}
\begin{split}
\lesssim 2^{-j(\frac 32+\delta)}\beta(r_1+r_2)& \Big(\Big|\int_0^1 e^{i2^j\phi(s)} (2^{-j}+|{\bf n}_s|)^{-\frac32-\delta} a(2^j|{\bf n}_s|)\;ds\Big|\\
&+\Big|\int_0^1 e^{i2^j\phi(s)} (2^{-j}+|{\bf n}_s|)^{-\frac32-\delta} a(2^j|{\bf n}_s|)\, \frac{b^2}{\frac{s^2}2+b^2}\;ds\Big|\\
&+\Big|\int_0^1 e^{i2^j\phi(s)} (2^{-j}+|{\bf n}_s|)^{-\frac32-\delta} a(2^j|{\bf n}_s|)\, \frac{sb^2}{\frac{s^2}2+b^2}\;ds\Big|\Big).
\end{split}
\end{equation}
Now we can bound
\begin{equation}
\begin{split}
\beta(r_1+r_2)& \Big(\int_0^1 \Big|\partial_s\big((2^{-j}+|{\bf n}_s|)^{-\frac32-\delta} a(2^j|{\bf n}_s|)\big)\Big|\;ds\\
&+\int_0^1 \Big|\partial_s\Big((2^{-j}+|{\bf n}_s|)^{-\frac32-\delta} a(2^j|{\bf n}_s|)\, \frac{b^2}{\frac{s^2}2+b^2}\Big)\Big|\;ds\\
&+\int_0^1\Big|\partial_s\Big((2^{-j}+|{\bf n}_s|)^{-\frac32-\delta} a(2^j|{\bf n}_s|)\, \frac{sb^2}{\frac{s^2}2+b^2}\Big)\Big|\;ds\Big)\lesssim 1,
\end{split}
\end{equation}
where we make the change of variables $s\to bs$ to avoid the issue {arising} when both $s, b\to 0$.
Then, by applying the Van der Corput Lemma (Lemma \ref{lem:VCL}) again, we obtain
\begin{equation}
\begin{split}
\Big|2^{-j(\frac 32+\delta)}\beta(r_1+r_2) \int_0^1 e^{i2^j\phi(s)}(2^{-j}+|{\bf n}_s|)^{-\frac32-\delta} &a(2^j|{\bf n}_s|)\frac{(-s+b^2)(\alpha s)}{\frac{s^2}2+b^2}\;ds\Big|\\
&\lesssim 2^{-j(\frac 32+\delta)} (2^jr_1r_2)^{-\frac12}.
\end{split}
\end{equation}
Therefore, we prove Lemma \ref{lem:ker-est12}.
\end{proof}

\begin{proof}[{\bf The proof of Lemma \ref{lem:ker-est3}:}] We first prove \eqref{kerDe3-est}. Arguing {similarly to} Lemma \ref{lem:ker-est12},
it suffices to show, when $r_1+r_2\sim 1$,
\begin{equation}
\begin{split}
\int_0^1\big|\partial_s \psi_{3, e}(r_1,r_2,\theta_1,\theta_2; s)\big|\, ds\lesssim 1,
\end{split}
\end{equation}
uniformly in $r_1,r_2,\theta_1,\theta_2$. For our purpose, when $r_1+r_2\sim 1$, we have
\begin{align}\nonumber
&\int_0^1\Big|\partial_s\Big[\Big((2^{-j}+|{\bf n}_s|)^{-\frac32-\delta} a(2^j|{\bf n}_s|)\Big)\\
&\qquad\Big(\frac{\sin(\theta_1-\theta_2+\pi)\cosh(\alpha s)}{\cosh(s)-\cos(\theta_1-\theta_2+\pi)}
-\frac{\sin(\theta_1-\theta_2+\pi)}{\frac{s^2}2+2\sin^2\big(\tfrac{\theta_1-\theta_2+\pi}2\big)}\Big)\Big]\Big|\,ds\lesssim 1\label{d-v2}
 ,
\end{align}
and
\begin{align}\nonumber
&\int_0^1\Big|\partial_s\Big[\Big((2^{-j}+|{\bf n}_s|)^{-\frac32-\delta} a(2^j|{\bf n}_s|)\\&\qquad-(r_1+r_2)^{-\frac32-\delta}a(2^j(r_1+r_2))\Big)\frac{\sin(\theta_1-\theta_2+\pi)}{\frac{s^2}2+2\sin^2
\big(\tfrac{\theta_1-\theta_2+\pi}2\big)}\Big]\Big|\, ds\lesssim 1,\label{d-v3}
\end{align}
which can be verified by modifying \cite[Lemma 5.1]{FZZ1} and by replacing $|{\bf n}_s|^{-\frac12}$ by $(2^{-j}+|{\bf n}_s|)^{-\frac32-\delta}$.

We next prove \eqref{kerDm3-est}. We will use the Morse Lemma to express the phase function as a quadratic form via a change of variables.
Let $$\bar{\varphi}(s)=\frac{\phi(s)-(r_1+r_2)}{r_1r_2}=\frac{\big(r_1^2+r_2^2+2r_1r_2\cosh s\big)^{\frac12}}{r_1r_2}-\frac{r_1+r_2}{r_1r_2},$$
then $\bar{\varphi}(0)=\bar{\varphi}'(0)=0$ and
$$\bar{\varphi}''(0)=\frac{1}{r_1+r_2}\neq 0.$$
Let
\begin{equation} \label{gs}
g(s)=\int_0^1(1-t)\bar{\varphi}''(ts)dt,
\end{equation}
then we can write $$\bar{\varphi}(s)=g(s)s^2.$$
We write
$$\phi(s)-(r_1+r_2)=r_1r_2\bar{\varphi}(s)=r_1r_2g(s)s^2.$$
Choose $\tilde{s}$ such that 
$$\frac{\tilde{s}^2}{r_1+r_2}=g(s)s^2,$$
then our aim is to estimate 
\begin{align*}
&\big|e^{-i2^j(r_1+r_2)}\tilde{K}_{D,m}^{3,j}(r_1,r_2;\theta_1-\theta_2)- H(r_1,r_2;\theta_1-\theta_2)\big|\\ \nonumber
&\leq 2^{-j(\frac 32+\delta)} \beta(r_1+r_2) \Big(\Big|\int_0^\infty e^{i2^jr_1r_2 g(s)s^2}\psi_{3, m}(s)\;ds-(r_1+r_2)^{\frac12}\int_0^\infty e^{i 2^j\frac{r_1r_2}{r_1+r_2} \tilde{s}^2} \psi_{3, m} (\tilde{s})\;d\tilde{s}\Big|\Big).
\end{align*}
Since $g(0)=\bar{\varphi}''(0)=(r_1+r_2)^{-1}\neq 0$, we make the change of
variables
$$\tilde{s}=\frac1{\sqrt{g(0)}}g(s)^{\frac12}s.$$
Then we obtain
\begin{equation}
\frac{d\tilde{s}}{ds}=\frac1{\sqrt{g(0)}}\Big(\frac12 g^{-\frac12}(s)g'(s) s+g^{\frac12}(s)\Big):=F(s).
\end{equation}
It suffices to estimate 
\begin{align*}
&\big|e^{-i2^j(r_1+r_2)}\tilde{K}_{D,m}^{3,j}(r_1,r_2;\theta_1-\theta_2)- H(r_1,r_2;\theta_1-\theta_2)\big|\\ \nonumber
&\leq 2^{-j(\frac 32+\delta)} \beta(r_1+r_2) \Big|\int_0^\infty e^{i2^jr_1r_2 g(s)s^2}\Big(\psi_{3, m}(s)-\psi_{3, m}(\tilde{s})F(s)\Big)\;ds\Big|\\
&\leq 2^{-j(\frac 32+\delta)} \beta(r_1+r_2) \Big(\Big|\int_0^1 e^{i2^jr_1r_2 g(s)s^2}\Big(\psi_{3, m}(s)-\psi_{3, m}(\tilde{s})F(s)\Big)\;ds\Big|\\
&\quad + \Big|\int_1^\infty e^{i2^jr_1r_2 g(s)s^2}\Big(\psi_{3, m}(s)-\psi_{3, m}(\tilde{s})F(s)\Big)\;ds\Big|\Big)\\
&:=I+II.
\end{align*}
We first consider $II$. When $s\in[1,+\infty)$,  there is no critical point for the phase function and
the derivative of the phase function $\phi=g(s)s^2$ is monotonic. By using Van der Corput Lemma \ref{lem:VCL}, we have
\begin{align}
II\lesssim 2^{-j(\frac 32+\delta)}\big(1+2^jr_1r_2\big)^{-1} .
\end{align}
We next consider $I$.
Furthermore, we observe that 
$$\bar{\varphi}'''(0)=0,\implies F'(0)=0,$$
then by the Taylor expansion, we have
$$F(s)=F(0)+F'(0)s+\frac12 F''(\theta s)s^2=1+\frac12 F''(\theta s)s^2, \quad 0\leq \theta\leq 1.$$
Therefore, we obtain
$$\Big(\psi_{3, m}(s)-\psi_{3, m}(\tilde{s})F(s)\Big)=\Big(\psi_{3, m}(s)-\psi_{3, m}(\tilde{s})\Big)-\frac12\psi_{3, m}(\tilde{s}) F''(\theta s)s^2.$$
By using Van der Corput Lemma \ref{lem:VCL} as before, the first term is bounded by
$$I\lesssim 2^{-j(\frac 32+\delta)}\big(1+2^jr_1r_2\big)^{-\frac12},$$
if we could verify that there exists a constant $C$ such that
\begin{align}\label{e:vercond1}
\Big| \big(\psi_{3, m}(1)-\psi_{3, m}(\tilde{s}_0)\big)\Big|+|\frac12\psi_{3, m}(\tilde{s}_0) F''(\theta )|\leq C, \quad \tilde{s}_0:=\frac1{\sqrt{g(0)}}g(s)^{\frac12}s\big|_{s=1},
\end{align}
and when $s\in[0, 1]$ and $\tilde{s}\in[0,\tilde{s}_0]$
\begin{align}\label{e:vercond1'}
&\int_0^1\Big|\frac{d}{ds}\Big(\psi_{3, m}(s)- \psi_{3, m}(\tilde{s})\Big)\Big|+\Big|\frac{d}{ds}\Big(\psi_{3, m}(\tilde{s}) F''(\theta s)s^2\Big)\Big|\, ds\leq C.
\end{align}
which can be verified as in \cite[Lemma 5.3]{FZZ1}.
Therefore, as desired, the difference is bounded by
$$\big|e^{-i2^j(r_1+r_2)}\tilde{K}_{D,m}^{3,j}(r_1,r_2;\theta_1-\theta_2)-H(r_1,r_2;\theta_1-\theta_2)\big|\lesssim 2^{-j(\frac 32+\delta)}\big(1+2^jr_1r_2\big)^{-\frac12}.$$
This completes the proof   of Lemma \ref{lem:ker-est3}.
\end{proof}

{\bf Conflicts of Interest and Data Availability Statement:}
The authors declare that there are no conflicts of interest relevant to the content of this manuscript. The research was conducted without any commercial or financial relationships that could be construed as a potential conflict of interest.
The data supporting the findings of this study are available from the corresponding author upon reasonable request.

\begin{center}

\end{center}

\end{document}